\documentclass[preprint]{elsarticle}
\usepackage{mathptmx}
\usepackage{amsmath}
\usepackage{amsthm}
\usepackage{natbib}
\usepackage[latin1]{inputenc}
\usepackage[T1]{fontenc}
\allowdisplaybreaks[2]
\usepackage{graphicx}
\usepackage{amssymb}
\usepackage{stmaryrd}
\usepackage[mathscr]{eucal}
\usepackage{hyperref} 

\newtheorem{theorem}{Theorem}
 \newtheorem{definition}{Definition}
 \newtheorem{lemma}{Lemma}

\newcommand{\Reg}{\Lambda}

\newcommand{\ud}{\mathrm{d}}

\def \N {\mathbb{N}}

\def \R {\mathbb{R}}

\def \ZnN {{0 \leq n < N}}

\providecommand{\Span}{\mathop{\mathrm{Span}}}
\providecommand{\fcont}{f}
\providecommand{\gcont}{g}
\providecommand{\Fcont}{F}
\providecommand{\Xcont}{X}
\providecommand{\Wcont}{W}
\providecommand{\cMcont}{\mathcal{M}}
\providecommand{\cCcont}{\mathcal{C}}
\providecommand{\cBcont}{\mathcal{B}}
\providecommand{\cDcont}{\mathcal{D}}
\providecommand{\Vcont}{V}
\providecommand{\cBcontTX}{\cBcont,\Xcont,T}
\providecommand{\cBcontTf}{\cBcont,\fcont,T}
\providecommand{\dimMcontcBcontTX}{\dim{\cMcont_{\cBcontTX}}}
\providecommand{\dimMcontcBcontTf}{\dim{\cMcont_{\cBcontTf}}}
\providecommand{\dimMcontO}{\dim{\cMcont_0}}
\providecommand{\dimMcontI}{\dim{\cMcont_I}}
\providecommand{\dimwidehatMcont}{\dim{\widehat{\cMcont}\,\,}}
\providecommand{\setoneKN}{\{1,\ldots,K_N\}}
\providecommand{\Proba}{\mathbb{P}}
\renewcommand{\dim}[1]{\mathop{\mathrm{dim}}\left( #1 \right)}
\providecommand{\dimMcont}{\dim{\cMcont}}

\title{Bandlet Image Estimation\\with Model Selection}

\author[bordeaux]{Ch. Dossal\corref{cor}}
\ead{charles.dossal@math.u-bordeaux1.fr}
\author[paris]{E. Le
    Pennec}
\ead{lepennec@math.jussieu.fr}
\author[cmap]{S. Mallat}
\ead{mallat@cmap.polytechnique.fr}
 \address[bordeaux]{IMB / Université de Bordeaux 1\\
351, cours de la Libération\\
33405 Talence Cedex / France}
\address[paris]{LPMA / Université Paris Diderot\\
175 rue du Chevaleret\\
75013 Paris / France}
\address[cmap]{CMAP / Ecole Polytechnique\\
 91128 Palaiseau Cedex France}

\cortext[cor]{Corresponding author}

\begin{document}

\begin{frontmatter}

\begin{abstract}
To estimate geometrically regular images in the white noise model and
obtain an adaptive near asymptotic minimaxity result, we consider a model selection
based bandlet
estimator. This bandlet estimator combines the best basis selection
behaviour of the model selection and the 
approximation properties of the bandlet dictionary.
We derive its  near  asymptotic minimaxity for geometrically regular images as an
example of model selection with general dictionary of orthogonal bases.
This paper is thus
also a self contained tutorial on model selection with orthogonal bases dictionary.

\end{abstract}
 
\end{frontmatter}

\section{Introduction}

A model selection based bandlet estimator has been  introduced 
by \citet{peyre07:_geomet_estim_orthog_bandl_bases}
to reduce white noise added to 
images having a geometrical regularity. This estimator projects the
observations on orthogonal bandlet vectors selected in a dictionary of 
orthonormal bases. This paper shows
that the risk of this estimator
is nearly asymptotically minimax for geometrically regular images. It
is also a tutorial on estimation with general dictionary of
orthogonal bases, through model selection. 
It explains with details how to build a thresholding estimator in an
adaptively chosen ``best'' basis and analyzes 
its
performance with the model selection approach of 
\citet{barron-risk-bounds}. 

Section~\ref{sec:image-estimation} describes the statistical setting of
the white noise model, and introduces the model of 
$\mathbf{C}^\alpha$ geometrically regular images. Images in this
class, originally proposed by \citet{KorostelevTsybakov},
are, roughly,
$\mathbf{C}^\alpha$ (Hölder regularity $\alpha$) outside a set of
$\mathbf{C}^\alpha$ curves in $[0,1]^2$. 
\citet{KorostelevTsybakov} prove that
 the minimax quadratic risk over
this class,  for a Gaussian white noise of variance
$\sigma^2$, has an asymptotic decay of the order of 
$\sigma^{2\alpha/(\alpha+1)}$.
They show  that the risk of any possible estimator cannot decay 
faster than this rate uniformly
for all functions of this class and exhibit an estimator that achieves
this rate. Their estimator relies on the knowledge of the regularity
exponent $\alpha$ and on an explicit detection of the contours, and is
not stable relatively to any image blurring.
Later, 
\citet{donoho-wedglets} 
overcomes the detection issue by
replacing it with an well-posed optimization problem. Nevertheless, both use a
model of images with sharp edges which limits their
applications 
since most image edges are not strict 
discontinuities. They are blurred because of various diffraction
effects which regularize discontinuities by unknown factors. 

The model selection based bandlet estimator, which can also be
described as
a thresholding estimator
in a best bandlet basis, does not have this restriction.
It does not rely on the detection of the precise localization of an
edge but only of a looser local direction of regularity.
Furthermore,
these directions of regularity are not estimated directly but
indirectly through a best orthogonal basis search algorithm which does
not require to know the regularity parameter $\alpha$.
Section~\ref{sec:proj-estim-model} gives a 
tutorial introduction of this type of estimators for arbitrary dictionary.
This generic class of thresholding estimators in a best basis selected in
a dictionary of orthonormal bases has been already studied by
\citet{donoho-ideal-denoising} and fit into the
framework of 
\citet{barron-risk-bounds},
\citep{BirgeMassart94lecam} and \citep{massart03:_concen_inequal_model_selec_saint_flour_notes}
This (self contained) section recalls the framework of these
estimators and their theoretical performance.
For the sake of completeness, a simplified proof of the main model
selection result is given in Appendix.

Section~\ref{sec:orth-band-basis} returns to the specific setting of
image processing and
applies the results of the previous section to geometric image
estimation. The choice of the representation (the choice of the dictionary of orthogonal
bases) becomes crucial and, after a short description of the bandlet bases,
 their use is
justified. The paper is concluded with Theorem~\ref{TheoFinal} which states
the adaptive near asymptotic minimaxity of the selection model based bandlet
estimator for geometrically regular images.

\section{Image estimation}
\label{sec:image-estimation}

\subsection{White noise model and acquisition}

During the digital acquisition process, a camera
measures an analog image $\fcont$
with a filtering and sampling process, which introduces
an additive white noise. 
In this white noise model, 
the process that is  observed can be written 
\[
\ud \Xcont_x = \fcont(x) \ud x + \sigma \ud
\Wcont_x,
\]
where $\Wcont_x$ is the Wiener process and $\sigma$ is a known noise
level parameter. 
This equation means that one is able to observe a Gaussian field $\Xcont_g$
indexed by functions $g\in L^2$ of mean
$E(\Xcont_g)=\langle f,g\rangle$
and covariance 
$E\left[\Xcont_{g}\Xcont_{g'}\right] = \langle g, g' \rangle$.

This model allows to consider asymptotics over $\sigma$ of a discrete
camera measurements process. Indeed, the measurement of a camera with
$N$ pixels can be modelled as the measure of $\Xcont_{\phi_{n}}$
over a family $\phi_n$ of $N$ impulse responses of the photo-sensors. Those
measurements,
\[
\Xcont_{\phi_n} = \langle \fcont , \phi_n \rangle + \sigma\, \Wcont_{\phi_n}
\text{ for }
0 \leq n < N 
\]
where $\Wcont_g$ is a Gaussian field of zero mean and covariance
$E\left[\Wcont_{g}\Wcont_{g'}\right] = \langle g, g' \rangle$,
define a ``projection'' of our observation $\ud\Xcont$ on the space
$\Vcont_N$, spanned by the $\phi_n$, that we  denote
$P_{\Vcont_N}\Xcont$. The white noise model allows to modify the
resolution of the camera depending on the noise level.
To simplify explanations, 
in the following we suppose that $\{ \phi_n \}_{0 \leq n < N}$
is an orthogonal basis, with no loss of generality, and thus that
\[
P_{\Vcont_N}\Xcont = \sum_{i=0}^N \Xcont_{\phi_n} \phi_n.
\] 

\subsection{Minimax risk and geometrically regular images}
We study the maximum risk of estimators for images $f$ in a given
class with respect to $\sigma$.
Model classes are often derived from
classical regularity spaces ($\mathbf{C}^\alpha$ spaces, Besov
spaces,...). This does not take into account the existence of geometrically
regular structures such as edges. This paper uses a geometric
image model appropriate for edges, but not for textures, where images are
considered as piecewise regular functions with discontinuities
along regular curves in $[0,1]^2$.
This geometrical image model has been proposed by 
\citet{KorostelevTsybakov} in their
seminal work on image estimation. It is used as a
benchmark to estimate or approximate images having some kind of geometric
regularity (%
\citet{donoho-wedglets},
\citet{shukla05:_rate},...). 
An extension of this model that incorporates a blurring
kernel $h$ has been proposed by \citet{bandlets-siam}
to model the various diffraction effects.
The resulting class of images studied in this paper is the set of
$\mathbf{C}^\alpha$ geometrically regular images specified by the
following definition.

\begin{definition}
\label{def:model}
A function $\fcont\in L^2([0,1]^2)$ is $\mathbf{C}^\alpha$ geometrically regular over $[0,1]^2$ if
\begin{itemize}
\item $ \fcont=\tilde{\fcont}$ or $ \fcont=\tilde{\fcont} \star h$
with $\tilde{\fcont} \in \mathbf{C}^{\alpha}(\Reg)$ for $ \Reg = [0,1]^2 - \{\mathcal{C}_\gamma \}_{1 \leq \gamma \leq G}$,
\item  the blurring kernel $h$ is  $\mathbf{C}^\alpha$, compactly supported 
in  $[-s,s]^2$ and 
$\|h\|_{\mathbf{C}^\alpha} \leq s^{-(2+\alpha)}$,
\item the edge curves $\mathcal{C}_\gamma$ are $\mathbf{C}^\alpha$
  and do not intersect tangentially if $\alpha>1$.
\end{itemize}
\end{definition}

\subsection{Edge based estimation}

\citet{KorostelevTsybakov} have built an
estimator that is asymptotically minimax for
geometrically regular functions $f$, as long as there is no blurring
and hence that $h = \delta$.
With a detection procedure, they partition the image in regions 
where the image is either regular or which include a
 ``boundary fragment'' corresponding to the subpart of a 
single discontinuity curve. 
In each region, they use either an estimator tailored to this  ``boundary
 fragments'' or a classical kernel estimator for the regular regions.
This yields a global estimate $\Fcont$ of the image $\fcont$.
If the $\fcont$ is $\mathbf{C}^\alpha$ outside the
boundaries and if the parametrization of the curve is also
$\mathbf{C}^\alpha$ then there exists a constant $C$ such that
\[
\forall \sigma~~,~~E\left[\| \fcont -\Fcont\|^2\right] \leq
C\sigma^{\frac{2\alpha}{\alpha+1}}\quad.
\]
This rate of convergence
achieves  the asymptotic minimax rate
for uniformly $\mathbf{C}^\alpha$ functions and thus the one for
 $\mathbf{C}^\alpha$ geometrically regular functions that includes this
class. This means that sharp edges do not alter
the rate of asymptotic minimax risk. However, this estimator 
is not adaptive relatively to the Holder exponent
$\alpha$ that must be known in advance.
Furthermore, it uses an edge detection procedure that fails when the image
is blurred or when the discontinuity jumps are not sufficiently large.

\citet{donoho-wedglets}  and 
\citet{shukla05:_rate} reuse the ideas of ``boundary fragment''
under the name ``horizon model'' to construct a piecewise polynomial
approximation of images. They derive efficient estimators
optimized for  $\alpha\in[1,2]$. 
These estimators use a
recursive partition of the image domain in dyadic squares, each square
being split in two parts by an edge curve that is a straight
segment. Both optimize the recursive partition and the choice of
the straight edge segment in each dyadic square by
minimizing a global function.
This  process leads to an asymptotically minimax estimator up to a logarithmic factor which 
is adaptive relatively to the Holder exponent as long
as $\alpha\in[1,2]$.

\citet{KorostelevTsybakov} as well as 
\citet{donoho-wedglets} and \cite{shukla05:_rate}
rely on the sharpness of image edges in their estimators.
In both cases, the estimator
is chosen amongst a family of images that are discontinuous 
across parametrized edges, and these estimators
are therefore not appropriate when the image edges are
blurred. We now consider estimators that do not have this
restriction:  they project the observation on adaptive subspaces in
which blurred as well as sharp edges are well represented.
They rely on two ingredients: the existence of bases in which
geometrical images can be efficiently approximated and the existence of a
mechanism to select, from the observation, a good basis and a good subset of coefficients
onto which it suffices to project the observation to obtain a good
estimator. We focus first on the second issue.

\section{Projection Estimator and Model Selection}
\label{sec:proj-estim-model}

The projection estimators we study are decomposed in two steps. First a linear
projection reduces the dimensionality of the problem by projecting the signal
in a finite dimensional space. This first projection is typically performed
by the digital acquisition device. Then a non-linear projection estimator 
refines this projector by reprojecting the resulting finite dimensional observation in a space that is chosen depending upon this observation.
This non-linear projection is obtained with a thresholding in a best basis
selected from a dictionary of orthonormal bases. Best basis algorithms for
noise removal have been introduced by 
\citet{CoifmanWickerhauser}. As recalled by 
\citet{Candes},
their risks have already been studied by 
\citet{donoho-ideal-denoising} and are a  special case of the general
framework of model selection proposed  by 
\citet{BirgeMassart94lecam}. 
Note that 
\citet{Kolacz} have studied a similar
problem in a slightly different setting.
 We recall in this section the framework of model
selection and state a selection model theorem
(Theorem~\ref{theo:minimizationE}) that is the main statistical
tool to prove the performance on the model selection based bandlet estimator.
This section is intended as a self contained tutorial presentation
of these best basis estimators and their resulting risk
upper bounds and contains no new results. 
Nevertheless, a simple (novel) proof of the (simplified) main result is given in Appendix.

\subsection{Approximation space $\Vcont_N$ and further projection}

The first step of our estimators is a projection in a finite
dimension space $\Vcont_N$ spanned by an orthonormal family
$\{ \phi_n \}_{0 \leq n < N}$. The choice of the dimension $N$ and of
the space $\Vcont_N$ depends on the noise level $\sigma$ but should not
depend on the function $f$ to be estimated. Assume for now that
$\Vcont_N$ is fixed and thus that we observe  $P_{\Vcont_N}X$. This
observation can be decomposed into $P_{\Vcont_N} \fcont + \sigma \Wcont_{\Vcont_N}$ where 
$\Wcont_{\Vcont_N}$ is a finite dimensional white noise on $\Vcont_N$. 

Our final estimator is a reprojection of this observation $P_{\Vcont_N} \Xcont$ 
onto a subspace $\cMcont \subset \Vcont_N$ which may
(and will)
depend on the observation:
the projection based estimator $P_{\cMcont} P_{\Vcont_N} \Xcont =  P_{ \cMcont}  X$.
The overall quadratic error can be decomposed in three terms:
\[
\|\fcont - P_{ \cMcont}  \Xcont \|^2 =  \|  \fcont - P_{\Vcont_N}  \fcont \|^2  + 
\| P_{\Vcont_N} \fcont  - P_{\cMcont}  \fcont \|^2 + \sigma^2 \|
P_{\cMcont} W\|^2.
\]
The first term is a bias term corresponding to the first linear approximation error due
to the projection on  $V_N$, the second term is also a bias term which
corresponds to the non linear
approximation of $P_{\Vcont_N}\fcont$ on $\cMcont$ while the third
term is a ``variance'' term corresponding to the contribution of the noise on $\cMcont$.

The dimension $N$ of $\Vcont_N$ has to be chosen large enough so that
with high probability,
for reasonable $\cMcont$,
$\| \fcont - P_{\Vcont_N} \fcont \|^2 \leq \| P_{\Vcont_N} \fcont  -
P_{\cMcont} \fcont \|^2 + \|P_{\cMcont} \Wcont \|^2$.
From the practical point of view, this means that the acquisition device resolution is set so that
the first linear approximation error due to discretization is smaller than the 
second non linear noise related error. Engineers often set $N$ so that both terms are of the same
order of magnitude, to limit the cost in terms of storage and
computations. In our white noise setting, we will explain how to chose
$N$ depending on $\sigma$.

For a fixed $\Vcont_N$, in order to obtain a small error, we need to
balance between the two remaining terms.  A space $\cMcont$ of large
dimension may
reduce the second bias term but will increase the variance term, a
 space $\cMcont$ of small dimension does the opposite.
It is thus necessary to find 
a trade-off between these two trends,
and select a space $\cMcont$ to minimize the sum of those two terms.

\subsection{Model Selection  in a Dictionary of orthonormal bases}

We consider a (not that) specific situation in which the space $\cMcont$
is spanned by some vectors from some orthonormal bases of $\Vcont_N$. More precisely,
let $\cBcont = \{ \gcont_n \}_\ZnN$ be an orthonormal basis of
$\Vcont_N$, that may be different from $\{\phi_n\}$, we consider
space $\cMcont$ spanned by a sub-family 
$\{ \gcont_{n_k} \}_{1 \leq k \leq M}$
of $M$ vectors and the projections of our observation on those spaces
\[
P_\cMcont \Xcont = \sum_{k=1}^M \Xcont_{\gcont_{n_k}} \, \gcont_{n_k} .
\]
Note that this projection, or more precisely its decomposition in the basis
$\{\phi_n\}$,  can be computed easily from the decomposition of
$P_{\cMcont}\Xcont$ in the same basis.
 
Instead of choosing a specific single orthonormal basis $\cBcont$, we define
a dictionary $\cDcont_N$ which is a collection of orthonormal bases
in which we choose adaptively the basis used.
Note that some  bases of $\cDcont_N$ may have vectors in common.
This dictionary can thus also be viewed as set $\{ \gcont_n \}$ of $K_N \geq N$
different vectors,
that are regrouped to form many
different  orthonormal bases. 
Any collection of $M$ vectors from the same orthogonal basis $\cBcont \in \cDcont_N$
generates a space $\cMcont$ that defines a possible estimator
$P_\cMcont \Xcont$ of $\fcont$.
Let $\cCcont_N = \{ \cMcont_\gamma \}_{\Gamma_N}$ be the family of all such 
projection spaces. Ideally we
would like to find the space $\cMcont \in\cCcont_N
$ which minimizes $\|\fcont - P_\cMcont \Xcont\|$.
We want thus to choose a ``best'' model $\cMcont$ amongst a collection
that is we want to perform a model selection task.

\subsection{Oracle Model}

As a projection estimator yields an estimation error
\[
\|\fcont - P_{\cMcont} \Xcont \|^2  = \|f-P_{\Vcont_N}\|^2 +\|
P_{\Vcont_N}  - P_{\cMcont} \fcont \|^2 + \| P_{\cMcont} \Wcont \|^2  
=  \|f  - P_{\cMcont} \fcont \|^2 + \| P_{\cMcont} \Wcont \|^2  
,
\]
the expected error of such an estimator is given by
\[
E\left[\|\fcont - P_{\cMcont} \Xcont \|^2\right]  = 
\|f  - P_{\cMcont} \fcont \|^2 + 
\sigma^2 \dimMcont.
\]
The best subspace for this criterion is the one that realizes the best trade-off
between the approximation error  $\|f - P_{\cMcont} \fcont \|^2$
and the complexity of the models measured by $\sigma^2 \dimMcont$. 

This expected error cannot be computed in practise since we have a single
realization of $\ud\Xcont$ (or of $P_{\Vcont_N}\Xcont$) . To (re)derive the classical model selection
procedure of \citet{BirgeMassart94lecam},
 we first
slightly modify our problem by 
searching for a subspace
$\cMcont$ such that the estimation error obtained by projecting $P_{\Vcont_N}\Xcont$ on
this subspace is small with only an overwhelming probability.
As in all model selection papers, we  use an upper bound 
 of the estimation error  obtained from
an upper bound of the energy of the noise projected on $\cMcont$.
Each of the $K_N$ projections of the noise on the $K_N$ different
vectors in the bases of the dictionary $\cDcont_N$ is thus
 $ \Wcont_{\gcont_k} \gcont_k$.  Its law is 
a Gaussian random variable of variance $\sigma^2$ along the vector $\gcont_k$.  
A standard large deviation result
proves that the norms of $K_N$ 
such Gaussian random variables 
are bounded simultaneously by $T = \sigma \sqrt {2 \log K_N}$ with 
a probability that tends to $1$ when $N$ increases. 
Since 
the noise energy projected in $\cMcont$
is the sum of $\dimMcont$ squared dictionary noise coefficients, we get
$\| P_{\cMcont} \Wcont \|^2  \leq \dimMcont\,T^2$.
It results that
\begin{equation}
\label{erroup}
\|\fcont - P_{\cMcont} \Xcont \|^2  \leq \|  \fcont  - P_{\cMcont} \fcont \|^2 + \dimMcont\,T^2 .
\end{equation}
over all subspaces $\cMcont$
with a probability that tends to $1$ as $N$ increases.
The estimation error is small if $\cMcont$ is 
a space of small dimension $\dimMcont$
which yields a small approximation error $\|\fcont - P_{\cMcont} \fcont \|$.
We denote by $\cMcont_O \in \cCcont_N$ the space that minimizes
the estimation error upper bound (\ref{erroup})
\[
\cMcont_O = \arg \min _{\cMcont \in \cCcont_N} (\|  \fcont  - P_{\cMcont} \fcont \|^2 + \dimMcont\,T^2) .
\]
Note that this optimal space cannot be determined from the observation
$\Xcont$ since $\fcont$ is unknown. It is called the oracle space
, hence the $O$ in the notation, to remind this fact.

\subsection{Penalized empirical error}

To obtain an estimator, it is thus necessary to replace this oracle space by a
``best'' space obtained only from the observation $P_{\Vcont_N}\Xcont$ that 
yields  (hopefully) a small estimation error.
A first step toward this goal is to notice that since all the spaces
$\cMcont$ are included into $\Vcont_N$, minimizing
\[
\| \fcont  - P_{\cMcont} \fcont \|^2 + \dimMcont\,T^2
\]
is equivalent to minimizing
\[
\| P_{\Vcont_N}\fcont  - P_{\cMcont} \fcont \|^2 + \dimMcont\,T^2
\].
A second step is to consider
the crude estimation of $\|  P_{\Vcont_N}\fcont  - P_{\cMcont} \fcont \|^2$  given 
by the empirical norm 
\[
\|  P_{\Vcont_N}\Xcont  - P_{\cMcont} \Xcont \|^2 = \| P_{\Vcont_N} \Xcont  \|^2 - \|P_{\cMcont} \Xcont \|^2 .
\]
This may seem naive because estimating 
$\| P_{\Vcont_N} \fcont  - P_{\cMcont} \fcont \|^2$ with $\|  P_{\Vcont_N}\Xcont  -
P_{\cMcont} \Xcont \|^2$ yields a large error
\[
\|  P_{\Vcont_N}\Xcont  - P_{\cMcont} \Xcont \|^2  -  \|  P_{\Vcont_N}\fcont  - P_{\cMcont} \fcont \|^2  = 
(\|  P_{\Vcont_N}\Xcont  \|^2 - \|P_{\Vcont_N}\fcont \|^2) + (\|P_{\cMcont} \fcont \|^2 - \|P_{\cMcont} \Xcont \|^2) ,
\]
whose expected value is $(N-\dimMcont) \sigma^2$, with typically $\dimMcont \ll N$. 
However, most of this
error is in the first term on the right hand-side, which has no effect
on the choice of space $\cMcont$. This choice
depends only upon the second term and is thus only influenced by 
noise projected in the space $\cMcont$ of lower dimension $\dimMcont$.
The bias and the fluctuation of this term, and thus the choice of
the basis, are controlled by increasing the parameter $T$.

We define the best empirical projection estimator $P_{\widehat \cMcont}$ as the
estimator that minimizes the resulting empirical penalized risk:
\begin{equation}
\label{bestEstim}
\widehat \cMcont = \arg \min _{\cMcont \in \cCcont_N} \|  P_{\Vcont_N}\Xcont  - P_{\cMcont} \Xcont \|^2 + \dimMcont\,T^2
\end{equation}

\subsection{Thresholding in a best basis}
\label{sec:thresh-best-basis}

Finding the best estimator which minimizes (\ref{bestEstim}) may seem
computationally untractable because the number of possible spaces $\cMcont \in \cCcont$
is typically an exponential function of the number $K_N$ of vectors in $\cDcont_N$.
We show that this best estimator may however be found with a thresholding
in a best basis.

Suppose that we impose that $\cMcont$ are generated by a subset of vectors from
a basis $\cBcont \in \cDcont_N$. The following (classical) lemma proves that among
all such spaces, the best projection estimator is obtained with 
a thresholding at $T$.

\begin{lemma}
Among all spaces $\cMcont$ that are generated by a subset
 of vectors
of an orthonormal basis $\cBcont = \{ \gcont_n \}_\ZnN$ of $\Vcont_N$, the estimator which
minimizes $\|  P_{\Vcont_N} \Xcont  - P_{\cMcont} \Xcont \|^2 + \dimMcont\,T^2$ is the thresholding estimator
\begin{equation}
\label{threEstma}
P_{\cMcont_{\cBcontTX}} \Xcont =  \sum_{n, |\langle \Xcont , \gcont_n \rangle| > T} \langle \Xcont , \gcont_n \rangle\, \gcont_n .
\end{equation}
\end{lemma}

\begin{proof}
Let $\cMcont=\Span\{\gcont_n\}_{n\in I}$ with $I\subset[0,N)$,
as $\cBcont$ is an orthonormal basis,
\begin{align*}
\|  \Xcont  - P_{\cMcont} \Xcont \|^2 + \dimMcont\,T^2
& = \sum_{n \notin I} |\langle \Xcont , \gcont_m\rangle|^2 + \sum_{n
  \in I} T^2 
\end{align*}
which is minimal if $I=\{n, |\langle \Xcont , \gcont_n\rangle|^2 > T^2\}$.
\end{proof}

The thresholding estimator (\ref{threEstma}) projects $\Xcont$ 
in the space $\cMcont_{\cBcontTX}$ generated by the vectors $\{
\gcont_m \}_{|\langle \Xcont , \gcont_m \rangle| > T}$, the vectors of $\cBcont$ which produce
coefficients above threshold. This lemma implies that best projection
estimators are necessarily thresholding estimators in some basis. 
Minimizing $\|  P_{\Vcont_N}\Xcont  - P_{\cMcont} \Xcont \|^2 + \dimMcont\,T^2$ over $\cMcont \in \cCcont$ is thus
equivalent to find the basis $\widehat \cBcont$ of $\Vcont_N$ which minimizes the thresholding 
penalized empirical risk:
\[
\widehat \cBcont =  \arg \min_{\cBcont \in \cDcont_N} \|  P_{\Vcont_N}\Xcont  - P_{\cMcont_{\cBcontTX}} \Xcont \|^2 + \dimMcont\,T^2  .
\] 
The best space which minimizes the empirical penalized 
risk in (\ref{bestEstim}) is derived from a thresholding in the best basis
$\widehat \cMcont = \cMcont_{\widehat{\cBcont},T}$.

The following theorem, similar to the one obtained first by
\citet{barron-risk-bounds}, 
proves that the thresholding estimation error in the best basis
is bounded by the estimation error by projecting in the oracle space $\cMcont_O$, 
up to a multiplicative factor.

\begin{theorem}
\label{theo:minimizationE} There exists an absolute function
$\lambda_0(K)\geq \sqrt{2}$ and some absolute constants $\epsilon>0$ and $\kappa>0$ such that
if we denote $\cCcont_N = \{ \cMcont_\gamma \}_\Gamma$ 
 the family of projection spaces generated by some vectors in an orthogonal
basis of a dictionary $\cDcont_N$ and denote $K_N$ be the number of 
different vectors in $\cDcont_N$. Then
for any $\sigma > 0$, if we let 
$T = \lambda\, \sqrt{\log(K_N)}\,\sigma$ with $\lambda\geq \lambda_0(K_N)$,
then
for any  $\fcont\in L^2$,  the thresholding estimator $F=P_{\cMcont_{\widehat{\cBcont},X,T}} \Xcont$ in the best basis
\[
\widehat \cBcont =  \arg \min_{\cBcont \in \cDcont_N} \|  P_{\Vcont_N}\Xcont  -
P_{\cMcont_{\cBcontTX}} \Xcont \|^2 + \dimMcontcBcontTX \,T^2 
\]
satisfies
\begin{equation*}\label{BestModelE}
E\left[\|\fcont - F\|^2\right] \leq (1+\epsilon) 
\left( \min_{\cMcont \in \cCcont_N} 
\| \fcont -P_{\cMcont}\fcont\|^2+ \dimMcont\,T^2 \right)
+ \frac{\kappa}{K_N}\sigma^2 .
\end{equation*}
\end{theorem}
For the sake of completion,
we  propose in Appendix a simple proof of
Theorem~\ref{theo:minimizationE}, inspired by 
\citet{BirgeMassart94lecam},
 which requires  only 
a concentration lemma for the norm of the noise
in all the subspaces spanned by the $K_N$ generators of
$\cDcont_N$ but with worse constants:
$\lambda_0(K) = \sqrt{32+\frac{8}{\log(K)}}$,
$\epsilon=3$ and $\kappa=64$.
Note this Theorem can be deduced from
 \citet{massart03:_concen_inequal_model_selec_saint_flour_notes}
 with different (better) constant
(and for roughly
$\lambda_0(K)>\sqrt{2}$) using a more complex proof
based on subtle Talagrand's inequalities. 
It results that any bound on
\(
\min _{\cMcont \in \cCcont_N} \|  \fcont  - P_{\cMcont} \fcont \|^2 + \dimMcont\,T^2 ,
\)
gives a bound on the risk of the best basis estimator $F$.

To obtain a computational estimator, the minimization
\[
\widehat \cBcont =  \arg \min_{\cBcont \in \cDcont_N} \|  P_{\Vcont_N}
\Xcont  - P_{\cMcont_{\cBcontTX}} \Xcont \|^2 + \dimMcontcBcontTX\,T^2 \quad,
\]
should be performed with a number of operations typically proportional
to the number $K_N$ of vectors in the dictionary. This requires to
construct appropriate dictionaries of orthogonal bases. 
Examples of such dictionaries have been proposed by 
\citet{CoifmanWickerhauser} 
with wavelet packets or by 
\citet{Coifman91remarques}
 with
local cosine bases for signals having localized time-frequency structures. 
Next section reviews some possible dictionaries for images and 
recall the construction of
the dictionary of bandlet orthogonal
bases that is adapted to the estimation of geometrically regular images.

\section{Best basis image estimation and bandlets}
\label{sec:orth-band-basis}

\subsection{Thresholding in a single basis}

When the dictionary $\cDcont_N$
is reduced to a single basis $\cBcont$, and there is thus no basis
choice, Theorem~\ref{theo:minimizationE} clearly applies  and reduces to the classical
thresholding Theorem of \citet{donoho-shrinkage}. The corresponding
estimator is thus the
classical thresholding estimator which quadratic risk satisfies
\begin{align*}
E\left[\|\fcont - P_{\cMcont_{\cBcontTX}} \Xcont\|^2\right] &\leq (1+\epsilon) 
\left( \min_{\cMcont \in \cCcont_N} 
\| \fcont -P_{\cMcont}\fcont\|^2+ \dimMcont\,T^2 \right)
+ \frac{\kappa}{N}\sigma^2
\end{align*}
It remains ``only'' to choose which basis to use and how to define the space  $\Vcont_N$ with respect
to $\sigma$.

Wavelet bases provide a first family of estimators used commonly in
image processing.
Such a two dimensional wavelet basis is constructed from
two real functions, 
a one dimensional  wavelet
$\psi$ and a corresponding one dimensional scaling function $\phi$,
which are both dilated
and  translated:
\[
 \psi_{j,k}(x) = \frac{1}{2^{j/2}} \psi\left(\frac{x-2^jk}{2^j}
\right)
\text{  and  }
 \phi_{j,k}(x) = \frac{1}{2^{j/2}} \phi\left(\frac{x-2^jk}{2^j}
\right)\quad.
\]
Note that the index $j$ goes to $-\infty$ when the
wavelet scale $2^j$ decreases. For a suitable choice of $\psi$ and $\phi$,
the family $\{ \psi_{j,k}(x)\}_{j,k}$ is an orthogonal basis of
$L^2([0,1])$ and the following family constructed by tensorization
\begin{align*}
\label{2dwavebs}
\left\{ 
\begin{array}{ccc}
\psi_{j,k}^V(x)=\psi_{j,k}^V(x_1,x_2)=
\phi_{j,k_1} (x_1)\,\psi_{j,k_2} (x_2),\\
\psi_{j,k}^H(x)=\psi_{j,k}^H(x_1,x_2)=\psi_{j,k_1} (x_1)\,\phi_{j,k_2} (x_2),\\
\psi_{j,k}^D(x)=\psi_{j,k}^D(x_1,x_2)=\psi_{j,k_1} (x_1)\,\psi_{j,k_2} (x_2) 
\end{array}
\right\}_{(j,k_1,k_2) 
}
\end{align*}
is an orthonormal basis of the square $[0,1]^2$.
Furthermore, each space 
\[
V_j=\Span\{\phi_{j,k_1}(x_1)\phi_{j,k_2}(x_2)\}_{k_1,k_2},
\]
called approximation space of scale $2^j$, admits
$\{ \psi_{l,k}^o\}_{o,l\geq j,k_1,k_2}$ as an orthogonal basis.
The approximation space $\Vcont_N$ of the previous section
coincides with the classical wavelet approximation space 
$V_j$ when $N=2^{-j/2}$. 

A classical approximation result ensures that for any function $\fcont$
$\mathbf{C}^\alpha$, as soon as the wavelet has more than $\lfloor
\alpha \rfloor+1$ vanishing moments, there is a constant $C$ such that, for any $T$,
$\min_{\cMcont \in \cCcont_N} 
\| P_{\Vcont_N}\fcont -P_{\cMcont}\fcont\|^2+ \dimMcont\,T^2 \leq C (T^2)^{\frac{\alpha}{\alpha+1}}$, 
and, for any $N$, $\|P_{V_N}\fcont - \fcont\|^2 \leq C N^{-\alpha}$. 
For $N=2^{-j/2}$ with $\sigma^2=[2^j,2^{j+1}]$, 
Theorem~\ref{theo:minimizationE} thus implies
\[
E[\|\fcont-\Fcont\|^2] \leq C (|\log(\sigma)|\sigma^2)^{\frac{\alpha}{\alpha+1}}~.
\]
This is up to the logarithmic term the best possible rate for
$\mathbf{C}^\alpha$ functions. Unfortunately, wavelets bases do not
provides such an optimal representation for the $\mathbf{C}^\alpha$
geometrically regular functions specified by Definition \ref{def:model}. 
Wavelets fail to capture the geometrical regularity of edges: near
them, the wavelets coefficients remain large. As explained in
\citet{MallatBook},
by noticing that those edges contribute at scale $2^j$ to $O(2^{-j})$
coefficients of order $O(2^{j/2})$,
one verifies that the rate of convergence in a wavelet basis decays like
$(|\log(\sigma)|\sigma^2)^{1/2}$, which is far from the asymptotical minimax rate.

A remarkably efficient representation was introduced 
by 
\citet{CurvSurprisingly}.
Their curvelets  are not isotropic like wavelets but are 
more elongated along a preferential direction and have two vanishing
moments along this direction. 
They are dilated and translated like wavelets but they are also rotated. 
The resulting family of 
curvelets $\mathcal{C}=\{c_n\}_n$ is not a basis of $L^2([0,1]^2)$ but 
a tight frame of $L^2(\R^2)$. This implies, nevertheless, that for any $f\in L^2([0,1]^2)$
\[
\sum_{c_n\in\mathcal{C}} |\langle f,c_n \rangle|^2 = A \|f\|^2\quad
\text{with $A>1$}.
\]
Although this is not an orthonormal basis, the results of
Section~\ref{sec:proj-estim-model} can be extended to this setting.
Projecting the data on the first $N=\sigma^{-1/2}$  curvelets with
significant intersection with the unit square
and thresholding the remaining coefficients with a threshold
$\lambda\sqrt{\log{N}}\sigma$ yields an estimator $F$ that satisfies
\[
E\left[\|\fcont-\Fcont\|^2\right] \leq C (|\log \sigma| \sigma^2)^{\frac{\alpha}{\alpha+1}}
\]
with a constant $C$ that depends only on $f$. This is the optimal
decay rate for the risk up to the logarithmic factor for $\alpha\in[1,2]$. 
No such fixed representation is known to achieve a similar result
for $\alpha$ larger than $2$.

\subsection{Dictionary of orthogonal bandlet bases}

To cope with a geometric regularity of order $\alpha>2$, one needs
basis elements which are more anisotropic than  the curvelets, are more
adapted to the geometry of edges and have more vanishing moments in
the direction of regularity.
Bandlet bases\citep{bandlets-ieee,bandlets-siam,peyre-bandlets-theo}
are orthogonal bases whose elements have such
properties.
Their construction is based on the observation that even if the
wavelet coefficients are large in the neighbourhood of an edge, these
wavelets coefficients are regular along the direction of the edge as
illustrated by Fig~\ref{fig1}.

\begin{figure}
\begin{center}
\includegraphics[width=\textwidth]{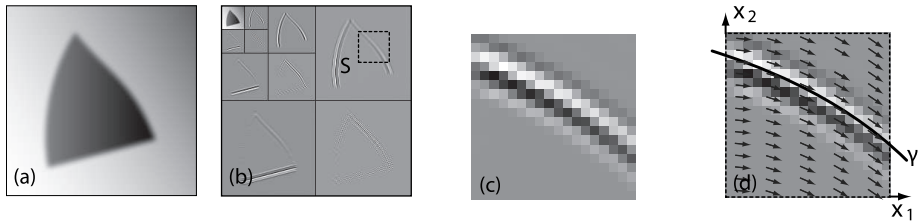}
\caption{
a) a geometrically regular image, 
b) the associated wavelet coefficients,
c) a close-up of wavelet coefficients in a detail space $W_j^o$ that shows their
remaining regularity, 
d) the geometrical flow adapted to this square of coefficients, here
it is vertically constant and parametrized by a polynomial curve $\gamma$}
\label{fig1}
\end{center}
\end{figure}
To capture this geometric regularity, 
the key tool is a local orthogonal
transform, inspired by the work of \citet{alpert-discrete}, that combines locally the wavelets
along the direction of regularity, represented by arrows in the
rightmost image of Fig~\ref{fig1}), to produce a new
orthogonal basis, a bandlet basis.
 By construction, the bandlets
are  elongated along the direction of
regularity and have the vanishing moments along this direction.
The (possibly large) wavelets coefficients are thus locally recombined along this
direction, yielding more coefficients of small amplitudes than before.

\begin{figure}
\begin{center}
\includegraphics[width=\textwidth]{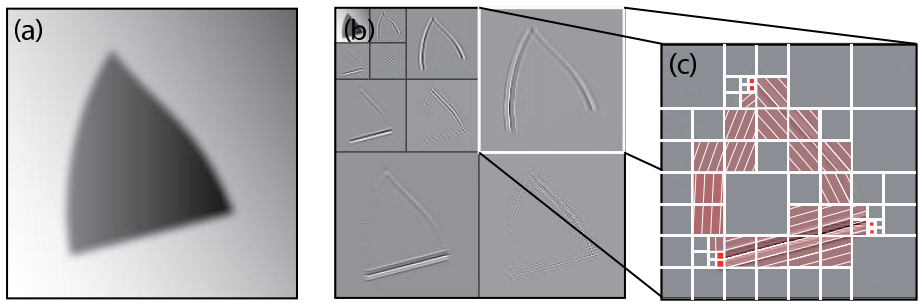}

\caption{
a) a geometrically regular image
b) the corresponding wavelet coefficients
c) the quadtree associated to the segmentation of a detail space $W_j^o$. In
each square where the image is not uniformly regular, the flow
is shown.}
\label{fig2}
\end{center}
\end{figure}

More precisely,
the construction of a bandlet basis 
of a  wavelet multiresolution space
$V_j=\Span \{ \phi_{j,k_1,k_2} \}_{k_1,k_2}$ 
starts by decomposing this space into detail wavelet spaces
\[
V_j = \bigoplus_{o,l>j} W_{l}^o~~~\mbox{with}~~~W_{l}^o = \Span \{
\psi_{l,k_1,k_2}^o\}_{k_1,k_2}~.
\]
For any level $l$ and orientation $o$,
the detail space $W_l^o$ is a space of dimension $(2^{-l})^2$.
Its coefficients are recombined using the Alpert transform induced
by some directions of regularity. This geometry is specified by a
local geometric flow, 
a vector field meant to follow the geometric direction of regularity. 
This geometric flow is further constraint to have
 a specific structure as illustrated in Fig.~\ref{fig2}, 
It is structured by a partition into dyadic squares
in which the
flow, if there exists, is vertically or horizontally constant.
In each square of the partition, the flow being thus easily parametrized by its tangent.

For each choice of geometric flow, a specific orthogonalization process \citep{peyre-bandlets-theo} 
yields an orthogonal basis of bandlets that have vanishing moments along the direction of the
geometric flow. 
This geometry should obviously be adapted to each image: the
partition and the flow direction should match the image structures.
This choice of geometry can be seen 
as an ill posed problem of estimation of
the edges or of the direction of regularity. To avoid this issue,
the problem is recasted as
a best basis
search in a dictionary. The geometry chosen is the one of the best
basis.

The first step is to define a dictionary
 $\cDcont_{(2^{-j})^2}$ of orthogonal bandlet bases of
$V_j$ or equivalently a dictionary of possible geometric flows.
Obviously this dictionary should be finite and this require 
 a discretization of the
geometry. As proved by \citet{peyre-bandlets-theo}, this is not an issue:
the flow does not have to follow exactly the
direction of regularity but only up to a sufficient known precision.
It is indeed sufficient to parametrize
the flow in any dyadic square
by the tangent of a polynomial of degree $p$ (the number
of vanishing moments of the wavelets). The coefficients of this
polynomial can be further quantized. The resulting family of geometric
flow in a square is of size $O(2^{-jp})$. 

A basis of the dictionary $\cDcont_{(2^{-j})^2}$ 
is thus specified by a set of
dyadic squares partitions for each details spaces $W_l^o$, $l>j$,
and, for each square of the partition, a flow parametrized by a direction and one of these
$O(2^{-jp})$ polynomials.
The number of bases in the dictionary
$\cDcont_{(2^{-j})^2}$  grows exponentially with $2^{-j}$, but 
 the total
number of different bandlets $K_{(2^{-j})^2}$ grows only polynomially like
  $O(2^{-j(p+4)})$.
Indeed the bandlets in a given dyadic square with a given geometry are
reused in numerous bases. The total number of bandlets in the
dictionary is thus bounded by the sum over all  $O(2^{-2j})$ dyadic squares
and all  $O(2^{-jp)})$ choices for the flow
of the number of bandlets in the square. Noticing that $(2^{-j})^2$ is
a rough bound of the number of bandlets in any subspaces of $V_j$, we obtain
the existence of a constant $C_K$ such that $2^{-j(p+4)}\leq
K_{(2^{-j})^2} \leq C_K 2^{-j(p+4)}$.

\subsection{Approximation in bandlet dictionaries}

The key property of the bandlet basis
dictionary is that it provides an asymptotically optimal representation of
$C^{\alpha}$ geometrically regular functions.
Indeed \citet{peyre-bandlets-theo} prove
\begin{theorem}
\label{theo:bandapprox}Let $\alpha < p$ where $p$ in the number of 
wavelet vanishing moments, for any $\fcont$ $\mathbf{C}^\alpha$
geometrically 
regular function, there exists a real number
$C$ such that for any $T>0$ and  $2^j\leq T$
\begin{equation}\label{Lagrangian1}
 \min_{\cBcont\in \cDcont_{(2^{-j})^2}} \|
 f-P_{\cMcont_{\cBcontTf}}f\|^2+\dimMcontcBcontTf T^2
\leqslant CT^{2\alpha/(\alpha+1)}
\end{equation}
where the
subspace $\cMcont_{\cBcontTf}$ is the space spanned by the vectors of $\cBcont$ whose
inner product with $\fcont$ is larger than $T$.
\end{theorem}
This Theorem gives the kind of control we require in
Theorem~\ref{theo:minimizationE}.

Being able to perform efficiently the minimization of the previous
Theorem is very important to exploit numerically this property. It
turns out that a fast algorithm
 can be used to 
find the best basis that minimizes $\|
\fcont-P_{\cMcont_{\cBcontTf}}\fcont\|^2+\dimMcontcBcontTf T^2$ or
equivalently
$\|
P_{V_j} \fcont-P_{\cMcont_{\cBcontTf}}\fcont\|^2+\dimMcontcBcontTf
T^2$.
We use first  the additive structure with respect to the subband
$W_l^o$ of this ``cost''  $\|
P_{V_j}\fcont-P_{\cMcont_{\cBcontTf}}\fcont\|^2+\dimMcontcBcontTf T^2$ to 
split the minimization into several independent minimizations
 on each
subbands. A bottom-top fast optimization of the geometry (partition and
flow) similar to the one proposed by 
\citet{CoifmanWickerhauser}, and
\citet{donoho-CART}
can
be performed on each subband thanks to two observations.
Firstly, for a given dyadic square, the limited number of possible flows is such that
the best flow can be obtained with a simple brute force exploration.
Secondly, the hierarchical tree structure of the partition and the
additivity of the cost function with respect to the partition
implies that
the best partition of a given dyadic square
is either itself or the union of the best partitions of its four dyadic
subsquares. 
This leads to a bottom up optimization algorithm once the
best flow has been found for every dyadic squares.
Note that this algorithm is adaptive with respect to $\alpha$: it does
not require the knowledge of the regularity parameter to be performed.

More precisely, the optimization algorithm goes as follows.
The brute force search of the best flow is
conducted independently over all dyadic squares and all detail spaces
with a total complexity of order $O(2^{-j(p+4)})$.
This yields a value
of the penalized criterion for each dyadic squares. 
It remains now to find the best partition. We proceed in a bottom up
fashion. The best partition with squares of width smaller than
$2^{j+1}$ is obtained from the best partition with squares of width
smaller than $2^j$: inside each dyadic square of
width $2^{j+1}$ the best partition is either the partition obtained so
far or the considered square.
This choice is made
 according to the cost computed so far.
Remark that the initialization is straightforward as
the best partition with square of size $1$
is obviously the full partition.
The complexity of this best partition search is of order $O(2^{-2j})$ and thus
the complexity of the best basis is driven by the best flow
search whose complexity is of order $O(2^{-j(p+4)})$, which
nevertheless remains
polynomial in $2^{-j}$.

\subsection{Bandlet estimators}

Estimating the edges is a complex task on blurred function and becomes
even much harder in presence of noise. Fortunately, the bandlet
estimator proposed by \citet{peyre07:_geomet_estim_orthog_bandl_bases} do not rely on such a detection process. The chosen
geometry is obtained with the best basis selection of the
previous section. This allows one to select an efficient basis even in the noisy setting.

Indeed, combining the bandlet approximation result of Theorem \ref{theo:bandapprox}
with the model selection results of Theorem \ref{theo:minimizationE} 
proves that the selection model based bandlet estimator  is near asymptotically minimax 
for $\mathbf{C}^\alpha$ geometrically regular images.

For a given noise level $\sigma$, one has to select a dimension
$N=(2^{-j})^2$
and a threshold $T$. The best basis algorithm selects then 
the bandlet basis $\widehat{\cBcont}$ 
amongst $\cDcont_{N}=\cDcont_{(2^{-j})^2}$ that minimizes
\[
\|P_{\Vcont_N}\Xcont-P_{\cMcont_{\cBcontTX}}\Xcont\|^2 + T^2 \dimMcontcBcontTX
\]
and the model selection based estimate is
$F=P_{\cMcont_{\cBcontTX}}\Xcont$.
We should now specify the choice of $N=(2^{-j})^2$ and $T$ in order to
be able to use Theorem~\ref{theo:minimizationE} and
Theorem~\ref{theo:bandapprox} to obtain the near asymptotic minimaxity of the
estimator. On the one hand, the dimension $N$ should be chosen large
enough so that the unknown linear approximation error
$\|f-P_{\Vcont_N}\|^2$ is small. One the other hand, the dimension $N$
should not be too large so that the total number of bandlets $K_N$,
which satisfies $\sqrt{N}^{(p+4)}\leq K_N \leq C_K \sqrt{N}^{(p+4)}$,
imposing a lower bound on the value of the threshold remains small.
For the sake of simplicity, as we consider an asymptotic behaviour,
we assume that $\sigma$ is smaller than $1/4$. This implies
that it exists $j<0$ such that $\sigma \in (2^{j-1}, 2^j]$
The following theorem proves that  
choosing $N=2^{-2j}$ and $T=\widetilde{\lambda}\sqrt{|\log
  \sigma|}\sigma$ with $\widetilde{\lambda}$
large enough yields a nearly asymptotically  minimax estimator.

\begin{theorem}\label{TheoFinal}
Let $\alpha < p$ where $p$ in the number of 
wavelet vanishing moments and
 let $K_0\in\N^*$ and $\widetilde{\lambda}\geq\sqrt{2(p+4)}
\sup_{K\geq K_0} \lambda_0(K)$.
For any $\mathbf{C}^\alpha$ geometrically regular function $f$,
there exists $C > 0$ such that for any 
\[
\sigma\leq \min( \frac{1}{4}, \max(C_K,K_0/2)^{-1/(p+4)}),
\]
if we let $N= 2^{-2j}$ with  $j$ such that $\sigma\in (2^{j-1}, 2^j]$ and 
$T=\widetilde{\lambda}\sqrt{|\log \sigma|}\sigma$, the estimator
$F=P_{\cMcont_{\widehat{\cBcont},\Xcont,T}}\Xcont$ obtained by thresholding
$P_{\Vcont_N}\Xcont$ with a threshold $T$ in the basis
$\widehat{\cBcont}$ of $\cDcont_{N}$ 
that minimizes
\[
\|P_{\Vcont_N}\Xcont-P_{\cMcont_{\cBcontTX}}\Xcont\|^2 + T^2 \dimMcontcBcontTX
\]
satisfies
\begin{equation*}
E \left[ \| f-F\|^2 \right] \leq C(|\log
\sigma|\sigma^2)^{\frac{\alpha}{\alpha+1}}.
\end{equation*}
\end{theorem}

Theorem~\ref{TheoFinal} is a direct consequence of
Theorem~\ref{theo:minimizationE} and Theorem~\ref{theo:bandapprox}, 
\begin{proof}
For any $\sigma\in(2^{j-1},2^j]$, 
observe that $2^{-j(p+4)}\leq K_{N}=K_{(2^{-j})^2}\leq C_K
2^{-j(p+4)}$ and thus
$
(2\sigma)^{-(p+4)} \leq K_{N} \leq C_K \sigma^{-(p+4)}.
$
The restriction on $\sigma$ further implies then that $K_N\geq K_0$
and $K_N\leq\sigma^{-2(p+4)}$. As $\widetilde{\lambda}\geq\sqrt{2(p+4)}
\sup_{K\geq K_0} \lambda_0(K)$,
$T= \widetilde{\lambda}\sqrt{|\log \sigma|}\sigma\geq
\lambda \sqrt{\log(K_N)} \sigma$ with 
 $\lambda\geq\lambda_0(K_N)$ so that Theorem~\ref{theo:minimizationE}
applies. This yields
\begin{align}
\label{eq:first}
E \left[ \| f-F\|^2 \right] \leq (1+\epsilon) \min_{\cMcont \in \cCcont_N} \left( \|f-P_{\cMcont}f\|^2 +
T^2 \dimMcont \right) + \frac{\kappa}{K_N} \sigma^2\quad.
\end{align}
Now as $T\geq 2^j$, Theorem~\ref{theo:bandapprox} applies and there is
a constant $C$ independent of $T$ such that
\[
\min_{\cMcont \in \cCcont_N} \left( \|f-P_{\cMcont}f\|^2 +
T^2 \dimMcont \right) \leq C (T^2)^{\alpha/(\alpha+1)}\quad.
\]
Plugging this bound into (\ref{eq:first}) gives the result. 
\end{proof}

The estimate $F=P_{\cMcont_{\widehat{\cBcont},T}}\Xcont$
is computed efficiently by the same fast algorithm used in the
approximation setting
 without requiring
the knowledge of the regularity parameter $\alpha$. The model selection
based bandlet estimator is
thus a tractable adaptive estimator that attains, up to the logarithmic term,
the best possible asymptotic minimax risk decay for $\mathbf{C}^\alpha$
geometrically regular function.

Although Theorem~\ref{TheoFinal} applies only to $\mathbf{C}^\alpha$
geometrically regular function, one can use  the bandlet estimator for any type of
images. 
Figure~\ref{fig-results-visuel} illustrates
the good behaviour of the bandlet estimator for natural images already shown in
\cite{peyre07:_geomet_estim_orthog_bandl_bases}. Each
line presents the original image, the degraded noisy image and two
estimations, one using classical translation invariant estimator \citep{coifman95:_trans}.
and the other using the bandlet estimator. The bandlet improvement
with respect to the classical wavelet estimator can be seen
numerically as well as
visually. The quadratic error is smaller with the bandlet estimator
and  the bandlets preserve much more geometric structures in the
images.

\begin{figure*}
\begin{center}
	\begin{tabular}{cccc}
\multicolumn{4}{c}{Geometric Image}\\
   \includegraphics[width=0.23\linewidth]{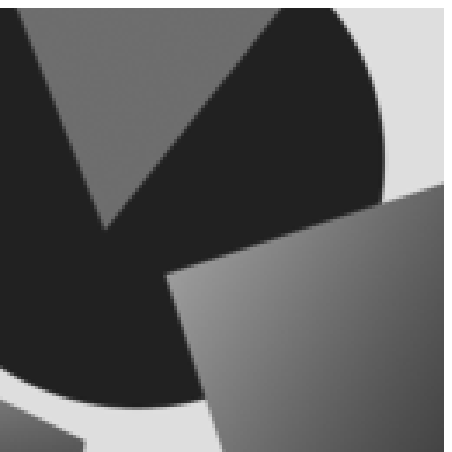}& 
   \includegraphics[width=0.23\linewidth]{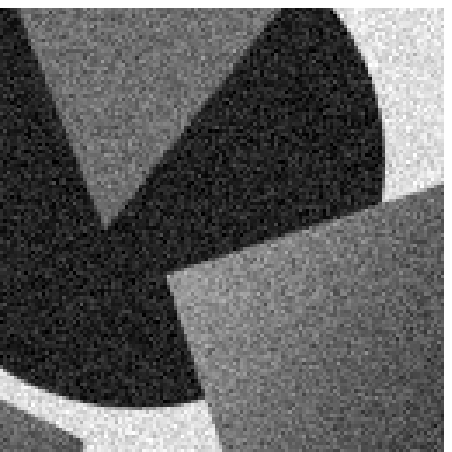}& 
   \includegraphics[width=0.23\linewidth]{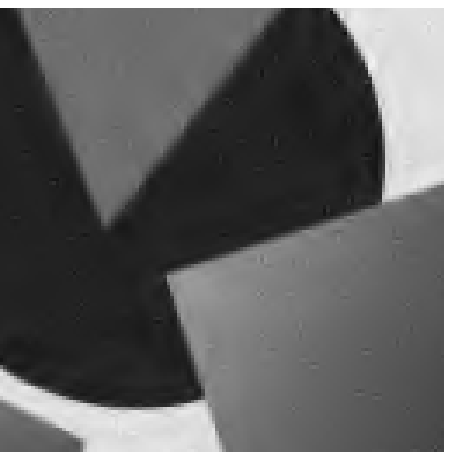}&
   \includegraphics[width=0.23\linewidth]{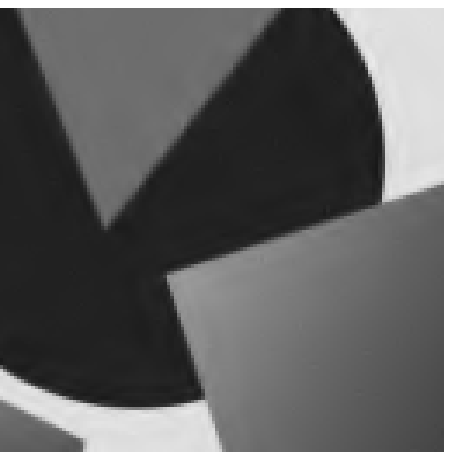}\\
   Original & Noisy (22.0dB) & TI Wavelets (36.0dB) & Bandlets
   (38.3dB) \\
\multicolumn{4}{c}{Barbara}\\
    \includegraphics[width=0.23\linewidth]{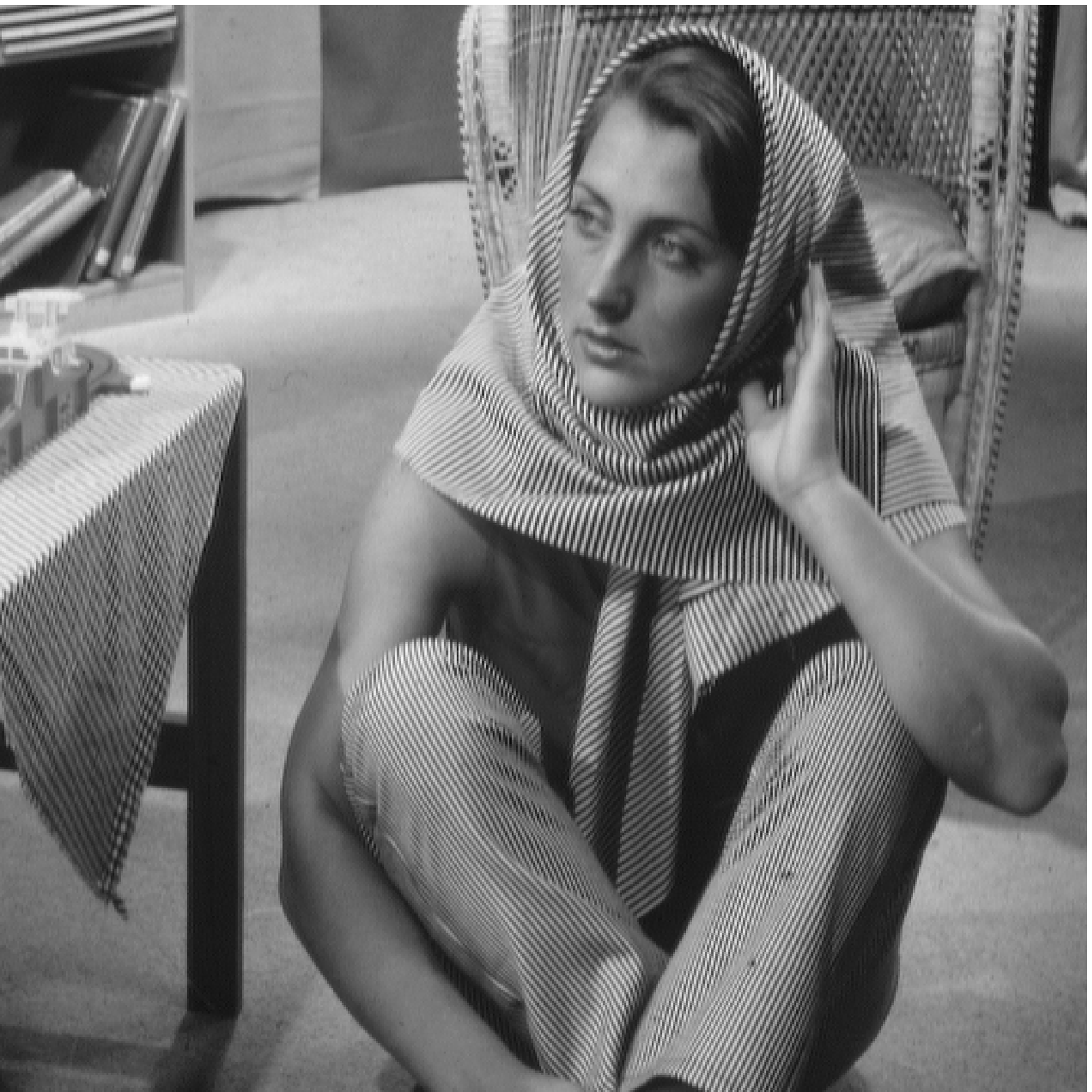}& 
	\includegraphics[width=0.23\linewidth]{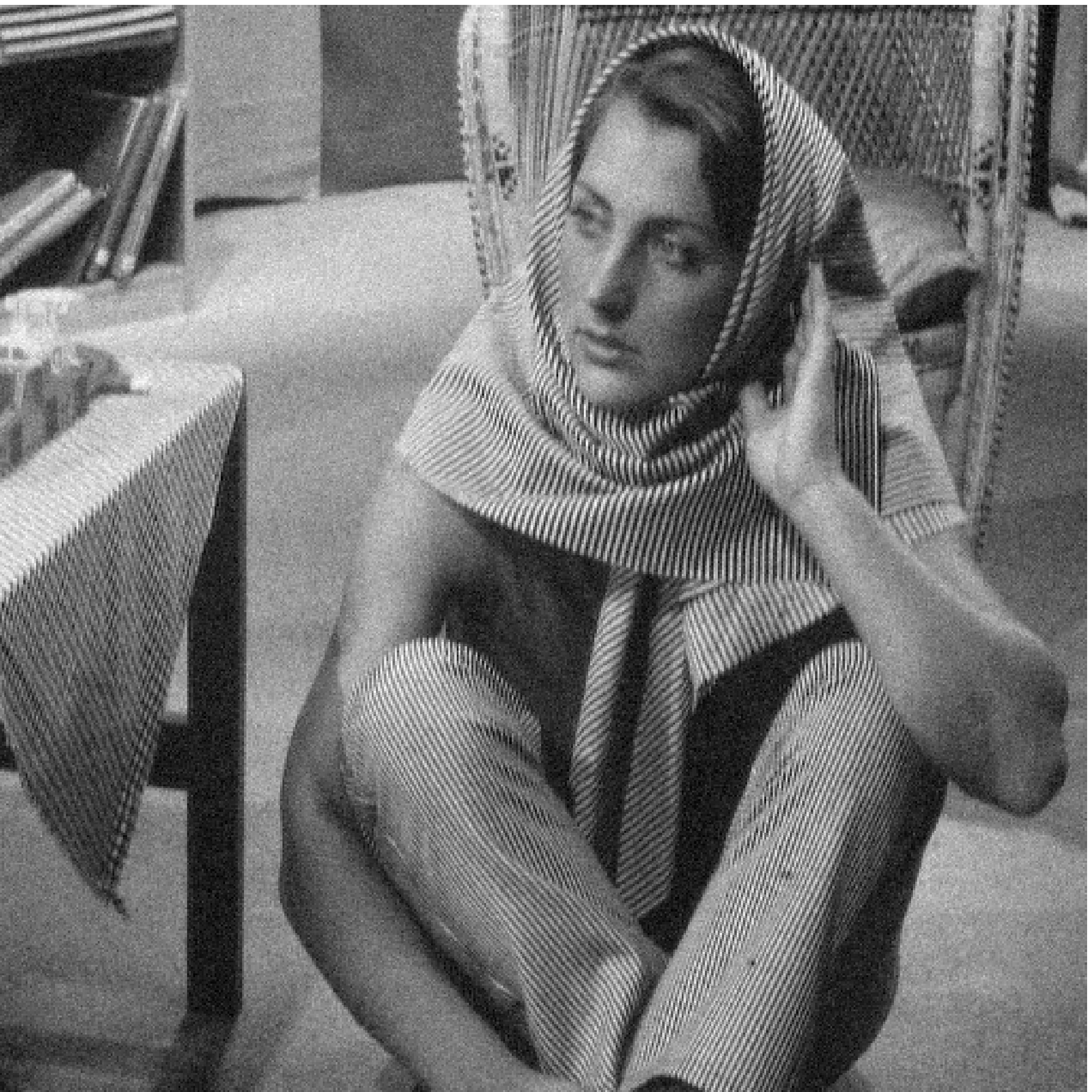}& 
	\includegraphics[width=0.23\linewidth]{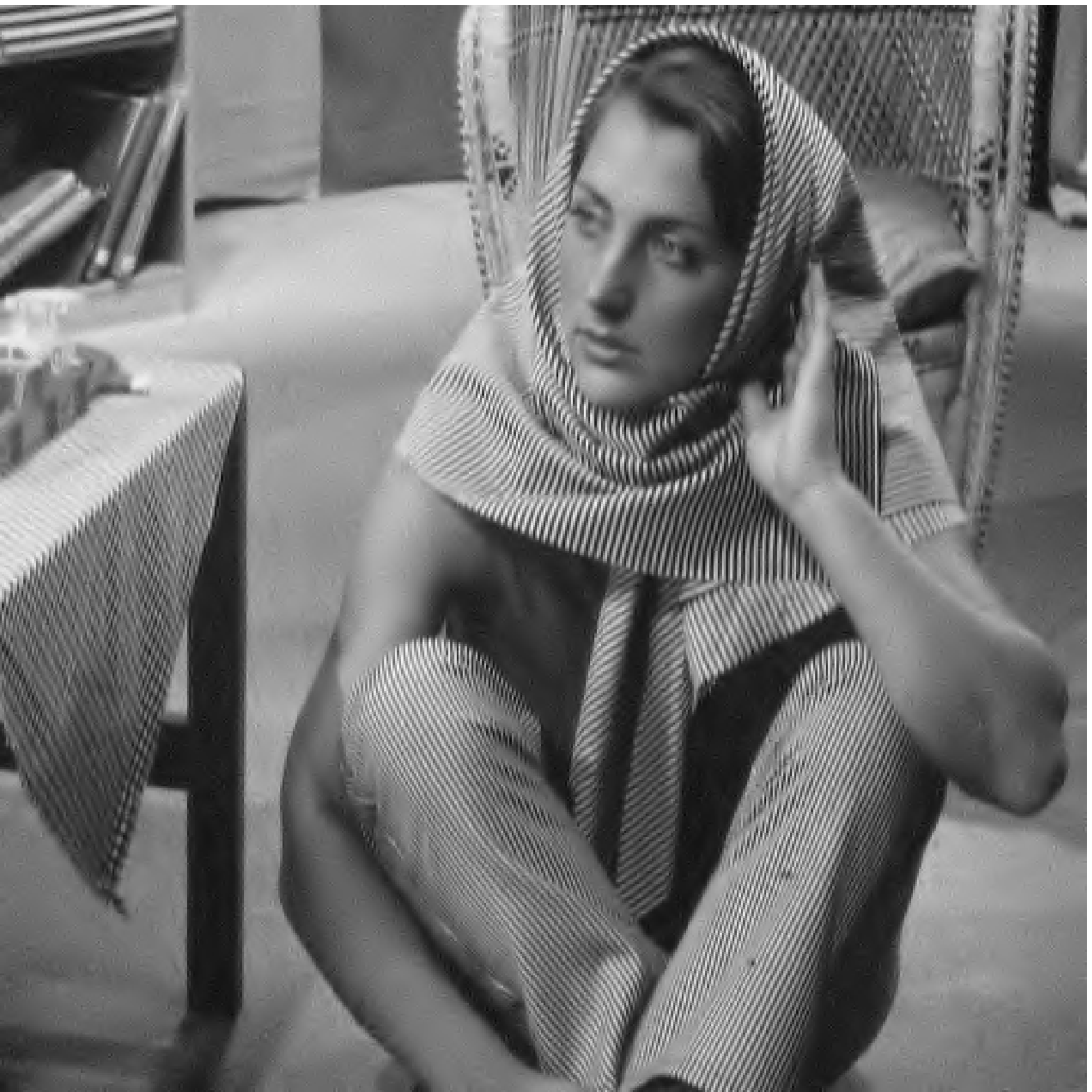}&
	\includegraphics[width=0.23\linewidth]{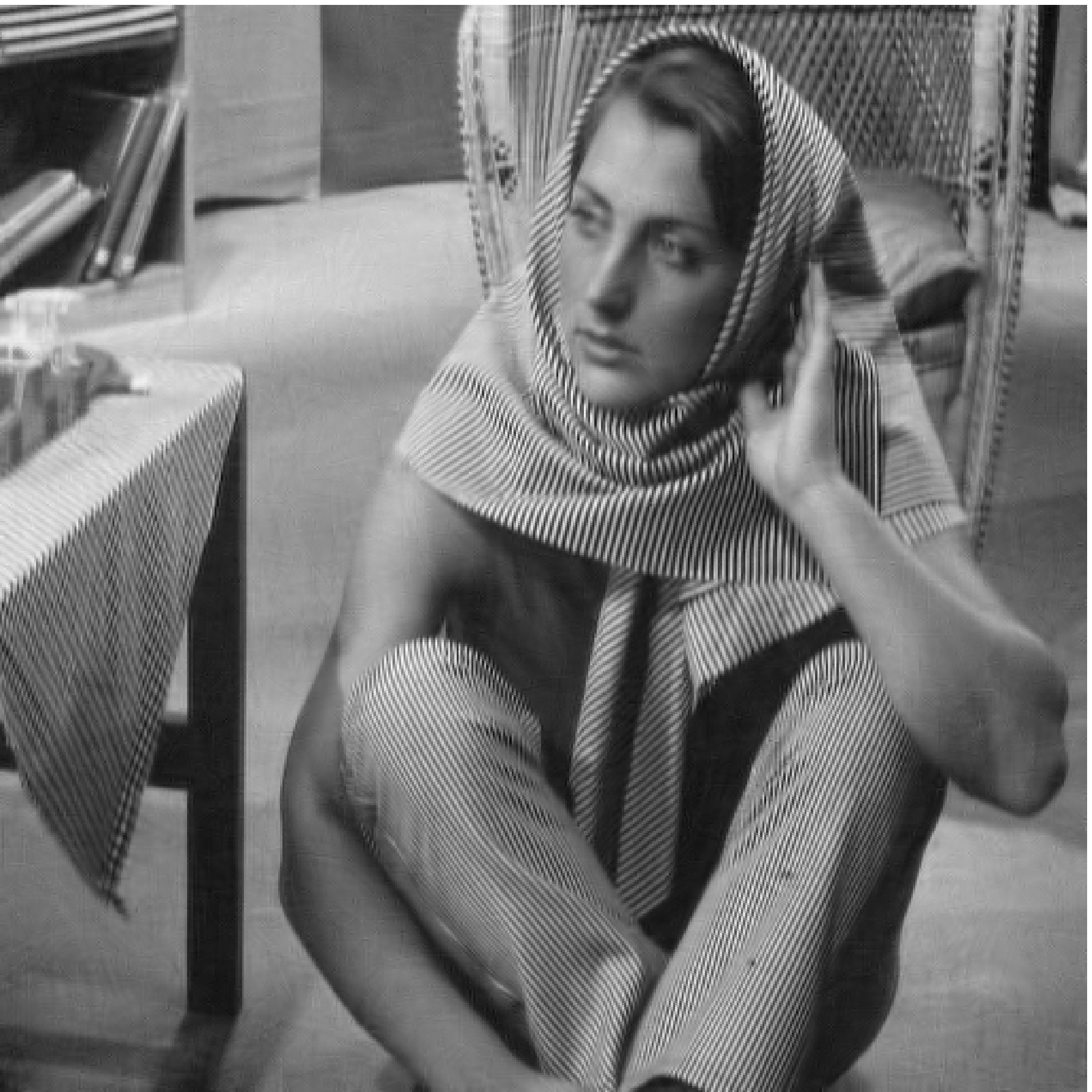}\\
   Original & Noisy (22.0dB) & TI Wavelets (26.5dB) & Bandlets
   (28.1dB)\\
\multicolumn{4}{c}{Barbara Closeup}\\
    \includegraphics[width=0.23\linewidth]{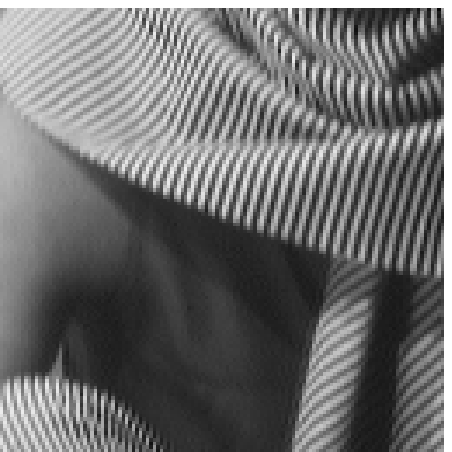}& 
	\includegraphics[width=0.23\linewidth]{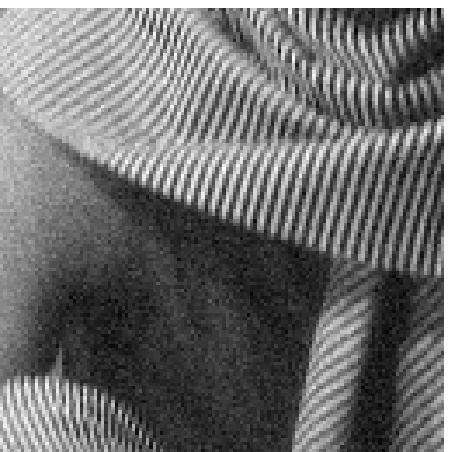}& 
	\includegraphics[width=0.23\linewidth]{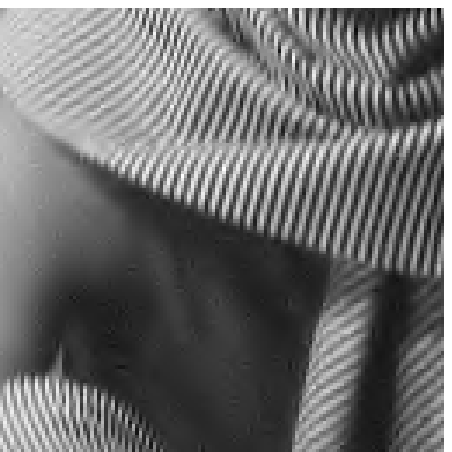}&
	\includegraphics[width=0.23\linewidth]{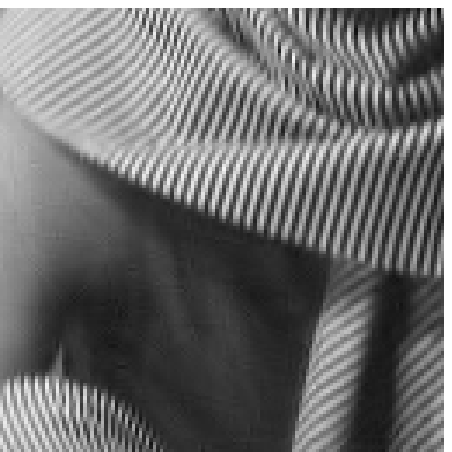}\\
   Original & Noisy (22.0dB) & TI Wavelets (26.5dB) & Bandlets
   (28.1dB)\\
\multicolumn{4}{c}{Lena}\\
    \includegraphics[width=0.23\linewidth]{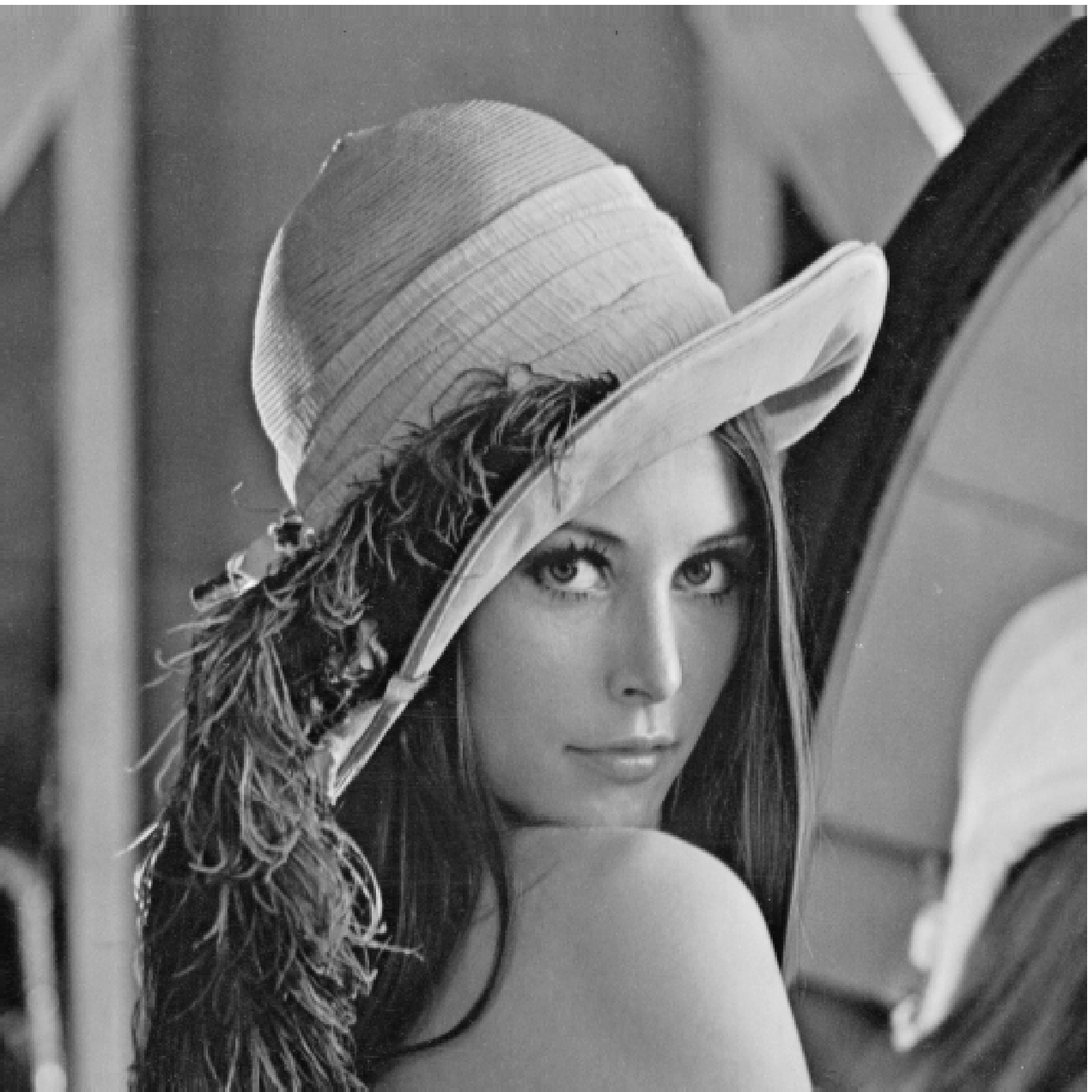}& 
	\includegraphics[width=0.23\linewidth]{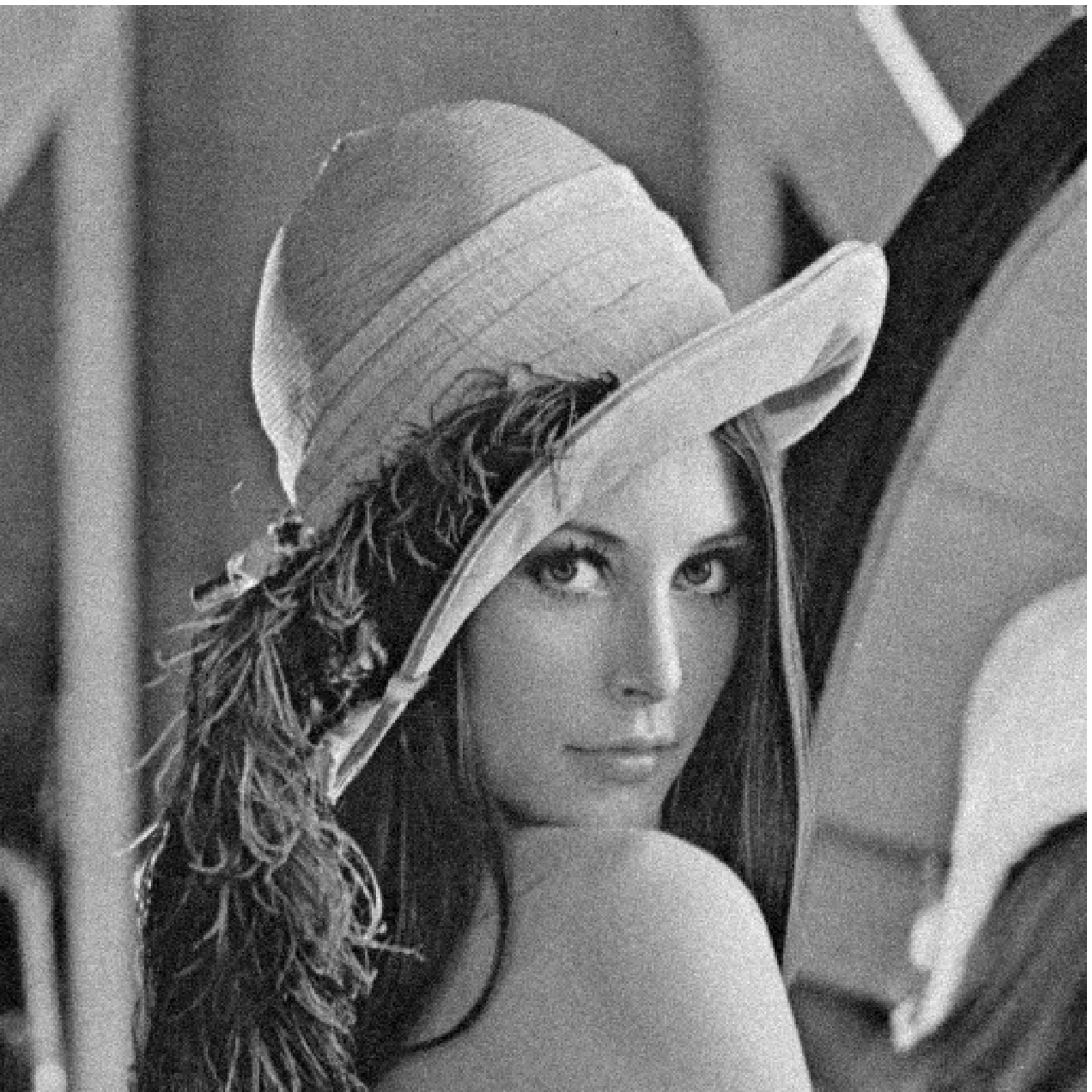}& 
	\includegraphics[width=0.23\linewidth]{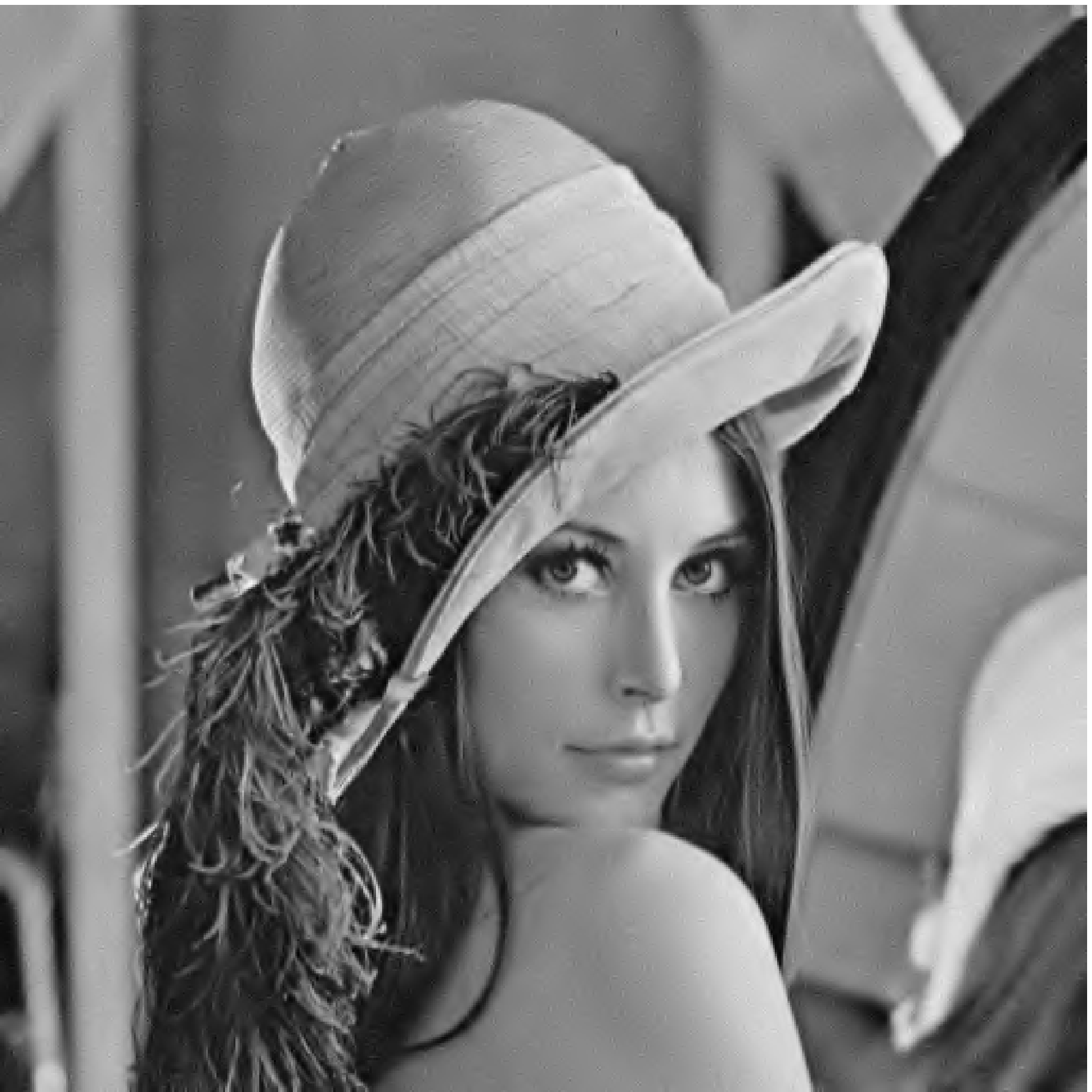}&
	\includegraphics[width=0.23\linewidth]{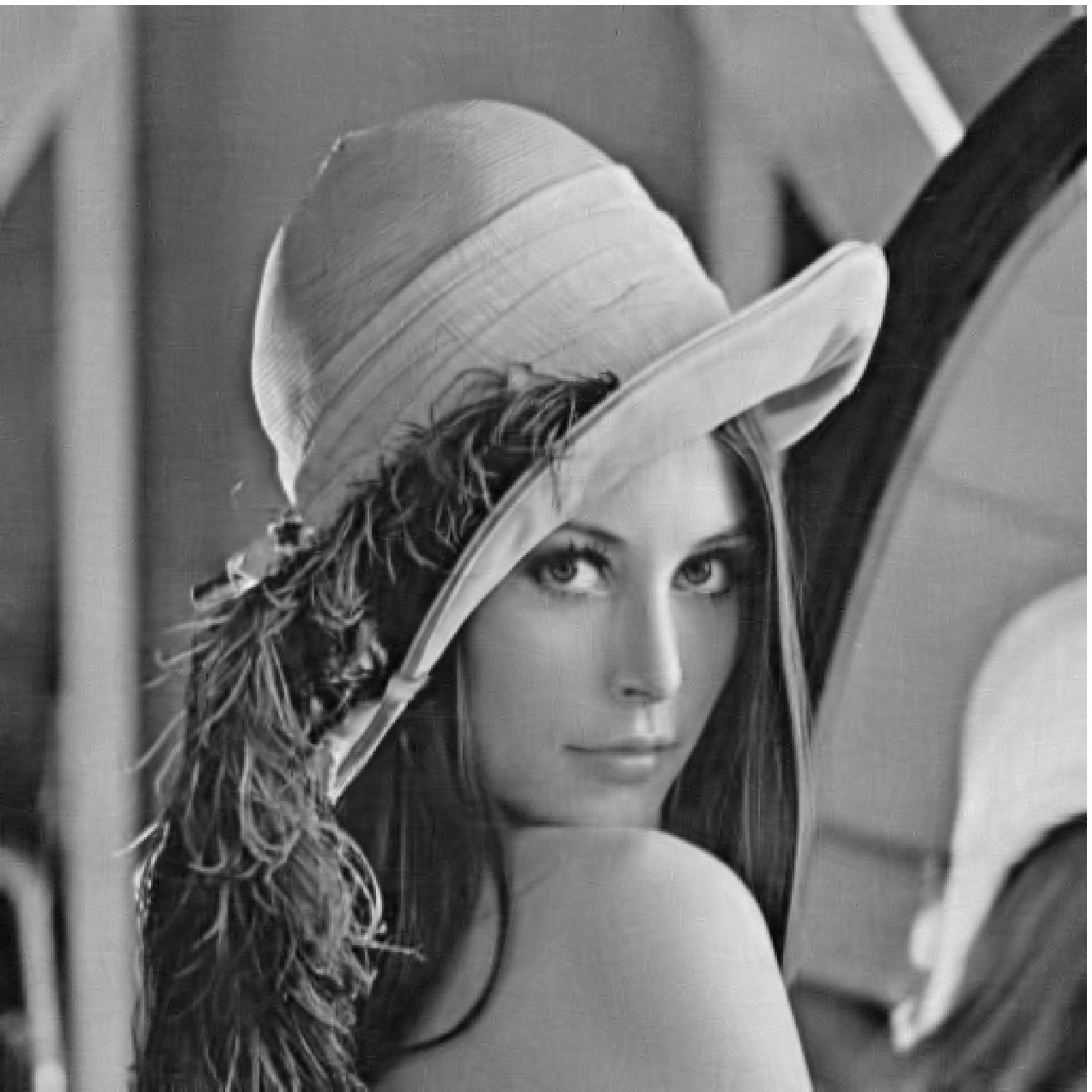}\\
   Original & Noisy (22.0dB) & TI Wavelets (28.1dB) & Bandlets
   (28.6dB)\\
\multicolumn{4}{c}{Lena Closeup}\\
    \includegraphics[width=0.23\linewidth]{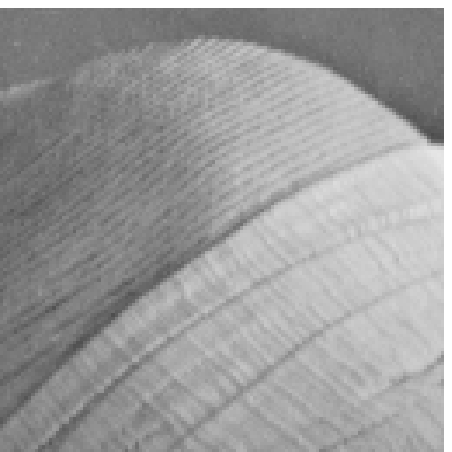}& 
	\includegraphics[width=0.23\linewidth]{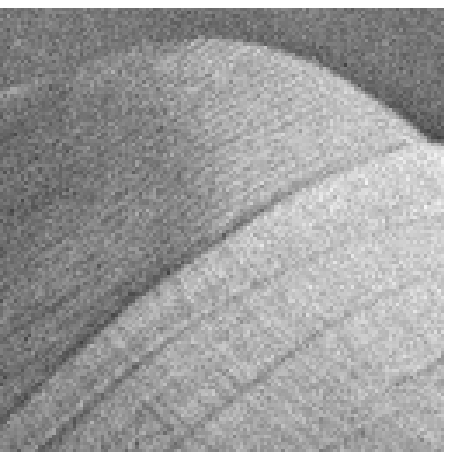}& 
	\includegraphics[width=0.23\linewidth]{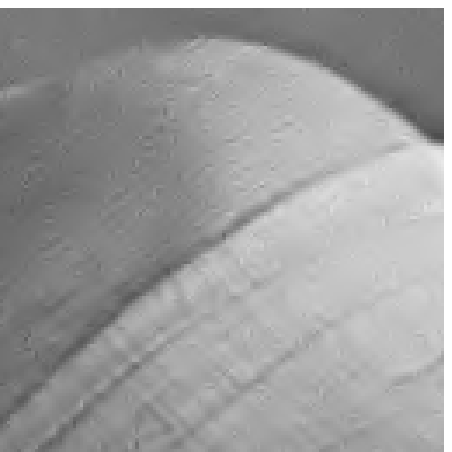}&
	\includegraphics[width=0.23\linewidth]{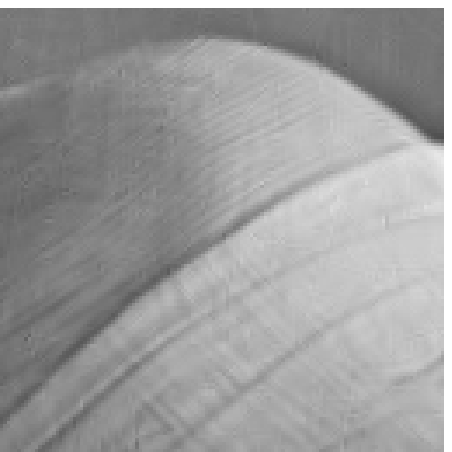}\\
   Original & Noisy (22.0dB) & TI Wavelets (28.dB) & Bandlets
   (28.6dB)
   \end{tabular}
 \end{center}

 \caption{
Comparison between  the  translation invariant wavelet
estimator and the bandlet estimator. The number within
parenthesis is the PSNR defined by
$-10\log\left(\frac{\|f-F\|^2}{\|f\|_\infty^2}\right)$ (the larger the better).
	}
\label{fig-results-visuel}
\end{figure*}

\appendix

\section{Proof of Theorem~\ref{theo:minimizationE}}

Concentration inequalities are at the core of all the selection model estimators.
Essentially, 
 the penalty should dominate the random fluctuation of the
minimized quantity. 
The key lemma, Lemma~\ref{lem:concentration}, uses
 a concentration inequality for Gaussian
variable to ensure, with high
probability, that the noise energy is small simultaneously in all
the subspaces $\cMcont_I$ spanned by a subset $I$ of the $K_N$ different
vectors, denoted by  $\gcont_k$, of $\cDcont_N$.
\begin{lemma}
\label{lem:concentration}
For all $u\geq 0$, with a probability greater than or equal to $1-2/K_N e^{-u}$,
\[
\forall I \subset \setoneKN \text{ and }\cMcont_I = \Span\{ \gcont_k \}_{k\in I},
\quad 
\|P_{\cMcont_I }\Wcont\| \leq \sqrt{M_I} + \sqrt{ 4 \log(K_N)
  \dimMcontI  + 2 u}
\] 
where $\dimMcontI$ is the dimension of $\cMcont_I$.
\end{lemma}

\begin{proof}
The key ingredient of this proof is a concentration
inequality. Tsirelson's Lemma\citep{tsirelson76:_norms_gauss} implies that for any $1$-Lipschitz function $\phi:
\mathbb{C}^n\to \mathbb{C}$  ($|\phi(x)-\phi(y)|\leq\|x-y\|$) if $\Wcont$
is a Gaussian standard white noise
in $\mathbb{C}^n$ then
\[
\Proba\left\{\phi(\Wcont)\geq E\left[\phi(\Wcont)\right]+t\right\}\leq e^{-t^2/2}\quad.
\]
  
For any space $\cMcont$, $\fcont \mapsto
\|P_{\cMcont }\fcont\|$
is $1$-Lipschitz. Note that one can first project $f$ into the finite
dimensional space $\Vcont_N$ without modifying the norm. We can thus apply Tsirelson's Lemma with  $t=\sqrt{4 \log
  (K_N) \dimMcont + 2 u }$
and obtain
\[
\Proba\left\{\|P_{\cMcont}\Wcont\| \geq 
E\left[\|P_{\cMcont}\Wcont\|\right]+\sqrt{
4    \log(K_N) \dimMcont + 2 u} \right\}\leq K_N^{-2\dimMcont}e^{-u}\quad.
\]
Now as $E\left[\|P_{\cMcont}\Wcont\|\right]\leq(E\left[\|P_{\cMcont}\Wcont\|^2\right])^{1/2} = \sqrt{\dimMcont}$, one derives
\[
\Proba\left\{\|P_{\cMcont}\Wcont\| \geq \sqrt{\dimMcont} + 
\sqrt{4 \log(K_N) \dimMcont + 2 u }\right\}\leq K_N^{- 2\dimMcont}e^{-u}\quad.
\]

Now
\begin{align*}
  \Proba\big\{\exists I\subset \setoneKN, 
 \|P_{\cMcont_I}\Wcont\| &\geq
\sqrt{\dimMcontI} + 
\sqrt{4 \log(K_N) \dimMcontI + 2 u }
\big\}
\\
&\leq \sum_{ I \subset \setoneKN}
  \Proba\left\{\|P_{\cMcont_I}\Wcont\| \geq
\sqrt{\dimMcontI} + 
\sqrt{4 \log(K_N) \dimMcontI + 2 u }
\right\}
\\ 
&\leq \sum_{I \subset \setoneKN}  K_N^{-2\dimMcontI }e^{-u}\\
& \leq \sum_{d=1}^{K_N} \binom{K_N}{d} K_N^{-2d}e^{-u}
\leq  \sum_{d=1}^{K_N} K_N^{-d}e^{-u} \\
& \leq \frac{K_N^{-1}}{1-K_N^{-1}}e^{-u}\\
\intertext{and thus}
  \Proba\big\{\exists I\subset \setoneKN, 
 \|P_{\cMcont_I}\Wcont\| &\geq
\sqrt{\dimMcontI} + 
\sqrt{4 \log(K_N) \dimMcontI + 2 u }
\big\}\\
&
  \leq \frac{2}{K_N}e^{-u}
\end{align*}
\end{proof}

The proof of Theorem~\ref{theo:minimizationE} follows from the
definition of the best basis, the oracle subspace and the previous Lemma.
\begin{proof}[Proof of Theorem~\ref{theo:minimizationE}]
Recall, that $P_{\Vcont_N}\Xcont = P_{\Vcont_N}\fcont + \sigma P_{\Vcont_N}\Wcont \in \Vcont_N$
with $P_{\Vcont_N}\Wcont$ a Gaussian white noise.
By construction, the thresholding estimate is 
$P_{\cMcont_{\widehat \cBcontTX}}\Xcont$ 
where
\[
\widehat \cBcont =  \arg \min_{\cBcont \in \cDcont_N} \|  P_{\Vcont_N}\Xcont  - P_{\cMcont_{\cBcontTX}} \Xcont \|^2 + \dimMcontcBcontTX\,T^2 \quad.
\]
To simplify the notation, we denote by $\widehat{\cMcont}$ and
$\dimwidehatMcont$ the corresponding space and its dimension.

Denote now $\dimMcontO$ the dimension of the oracle subspace $\cMcont_O$ that has been defined as the
minimizer of 
\[
\| P_{\Vcont_N}\fcont  - P_{\cMcont} \fcont \|^2 + \dimMcont \,T^2 \quad.
\]

By construction,
\begin{align*}
\| P_{\Vcont_N}\Xcont - P_{\widehat{\cMcont}}\Xcont\|^2+\lambda^2\log(K_N)\sigma^2 \dimwidehatMcont 
&\leq \| P_{\Vcont_N}\Xcont-P_{\cMcont_O}\fcont \|^2 + \lambda^2 \log(K_N)
\sigma^2 \dimMcontO.
\end{align*}
Using 
\[
\|P_{\Vcont_N}\Xcont - P_{\widehat{\cMcont}}\Xcont\|^2
=\|P_{\Vcont_N}\Xcont- P_{\Vcont_N}\fcont\|^2+\|P_{\Vcont_N}\fcont
 - P_{\widehat{\cMcont}}\Xcont\|^2+2\langle P_{\Vcont_N}\Xcont -
 P_{\Vcont_N}\fcont,P_{\Vcont_N}\fcont-P_{\widehat{\cMcont}}\Xcont\rangle
\] and a similar
  equality for $\| P_{\Vcont_N}\Xcont-P_{\cMcont_O}f \|^2$, one
  obtains
\begin{align*}
\begin{split}
\|P_{\Vcont_N}\fcont-P_{\widehat{\cMcont}}\Xcont\|^2+ \lambda^2\log(K_N)\sigma^2  \dimwidehatMcont  
&\leq \|P_{\Vcont_N}\fcont-P_{\cMcont_O}\fcont\|^2+\lambda^2\log(K_N)\sigma^2\dimMcontO\\
& \mspace{80mu}+2\langle P_{\Vcont_N}\Xcont- P_{\Vcont_N}\fcont, P_{\widehat{\cMcont}}\Xcont -
P_{\cMcont_O}\fcont
\rangle
\end{split}
\end{align*}

One should now focus on the bound on the scalar product :
\begin{align*}
| 2\langle P_{\Vcont_N}\Xcont- P_{\Vcont_N}\fcont&, P_{\widehat{\cMcont}}\Xcont -
P_{\cMcont_O}\fcont
\rangle|\\
& = |2 \langle\sigma
P_{\widehat{\cMcont} + \cMcont_O}\Wcont , P_{\widehat{\cMcont}}\Xcont -
P_{\cMcont_O}\fcont\rangle|\\
& \leq 2 \sigma \|P_{\widehat{\cMcont} + \cMcont_O}\Wcont\| 
(\|P_{\widehat{\cMcont}}\Xcont - P_{\Vcont_N}\fcont\| +\| P_{\Vcont_N}\fcont -
P_{\cMcont_O}\fcont\|)
 \\
\intertext{and, using Lemma~\ref{lem:concentration}, with a probability
  greater than or equal to $1-\frac{2}{K_N}e^{-u}$}
  & \leq 2 \sigma \left( \sqrt{ \dimwidehatMcont  +  \dimMcontO  } + 
  \sqrt{4 \log(K_N) ( \dimwidehatMcont  +  \dimMcontO  ) + 2 u }\right)\\
&\mspace{40mu}\times
(\|P_{\widehat{\cMcont}}\Xcont - P_{\Vcont_N}\fcont\| +\| P_{\Vcont_N}\fcont -
P_{\cMcont_O}\fcont\|)
\\
 \intertext{applying $2xy\leq \beta^{-2} x^2 + \beta^{2} y^2$
   successively with $\beta=\frac{1}{2}$ and $\beta=1$ leads to}
| 2\langle P_{\Vcont_N}\Xcont- P_{\Vcont_N}\fcont&, P_{\widehat{\cMcont}}\Xcont -
P_{\cMcont_O}\fcont
\rangle|\\
 & \leq 
 \left(\frac{1}{2}\right)^{-2} 2  \sigma^2 (  \dimwidehatMcont  +  \dimMcontO  + 4 \log(K_N) (  \dimwidehatMcont  +  \dimMcontO )
  + 2 u  )\\
&\mspace{40mu}
+ \left(\frac{1}{2}\right)^2 2 ( \|P_{\widehat{\cMcont}}\Xcont - P_{\Vcont_N}\fcont\|^2 +\| P_{\Vcont_N}\fcont -
P_{\cMcont_O}\fcont\|^2)
 \quad.
\end{align*}

Inserting this bound into
\begin{align*}
\begin{split}
\|P_{\Vcont_N}\fcont-P_{\widehat{\cMcont}}\Xcont\|^2+ \lambda^2\log(K_N)\sigma^2  \dimwidehatMcont  
&\leq
\|P_{\Vcont_N}\fcont-P_{\cMcont_O}\fcont\|^2+\lambda^2\log(K_N)\sigma^2
\dimMcontO 
\\
& \mspace{80mu}
+|2\langle P_{\Vcont_N}\Xcont- P_{\Vcont_N}\fcont, P_{\widehat{\cMcont}}\Xcont -
P_{\cMcont_O}\fcont
\rangle|
\end{split}
\end{align*}
yields
\begin{align*}
\frac{1}{2} \|P_{\Vcont_N}\fcont - P_{\widehat{\cMcont}}\Xcont \|^2
& \leq \frac{3}{2} \| P_{\Vcont_N}\fcont - P_{\cMcont_O}\fcont\|^2
+ \sigma^2 ( \lambda^2 \log(K_N) + 8 ( 1 + 4 \log(K_N) ))  \dimMcontO \\
& \mspace{40mu} + \sigma^2 ( 8 ( 1 + 4 \log(K_N) ) - \lambda^2
\log(K_N))  \dimwidehatMcont 
+ 16 \sigma^2 u 
\end{align*}
So that if $\lambda^2 \geq 32 + \frac{8}{\log(K_N)}$
\begin{align*}
\|P_{\Vcont_N}\fcont - P_{\widehat{\cMcont}}\Xcont \|^2
& \leq 3 \| P_{\Vcont_N}\fcont - P_{\cMcont_O}\fcont\|^2
+ 4 \sigma^2  \lambda^2 \log(K_N)   \dimMcontO 
+ 32 \sigma^2 u 
\end{align*}
which implies
\begin{align*}
\|P_{\Vcont_N}\fcont - P_{\widehat{\cMcont}}\Xcont \|^2
& \leq 4 ( \| P_{\Vcont_N}\fcont - P_{\cMcont_O}\fcont\|^2 + \sigma^2  \lambda^2 \log(K_N)   \dimMcontO ) + 32 \sigma^2 u 
\end{align*}
where this result holds with probability greater than or equal to $1-\frac{2}{K_N}e^{-u}$.

Recalling that this is valid for all $u\geq 0$, one has
\begin{align*}
\Proba\left\{ \|P_{\Vcont_N}\fcont - P_{\widehat{\cMcont}}\Xcont \|^2
 - 4 ( \| P_{\Vcont_N}\fcont - P_{\cMcont_O}\fcont\|^2 + \sigma^2  \lambda^2 \log(K_N)   \dimMcontO )
\geq  32 \sigma^2 u \right\} \leq \frac{2}{K_N}e^{-u}
 \end{align*}
which implies by integration over $u$
\[
E\left[
\|P_{\Vcont_N}\fcont - P_{\widehat{\cMcont}}\Xcont \|^2
- 4 ( \| P_{\Vcont_N}\fcont - P_{\cMcont_O}\fcont\|^2 + \sigma^2  \lambda^2 \log(K_N)   \dimMcontO )
\right] \leq 32 \sigma^2 \frac{2}{K_N} 
\]
that is the bound of Theorem~\ref{theo:minimizationE}
\[
E\left[
\|P_{\Vcont_N}\fcont - P_{\widehat{\cMcont}}\Xcont \|^2
\right] \leq 4 ( \|P_{\Vcont_N} \fcont - P_{\cMcont_O}\fcont\|^2 + \sigma^2
\lambda^2 \log(K_N)   \dimMcontO ) + 32 \sigma^2 \frac{2}{K_N} 
\]
up to $\|\fcont-P_{\Vcont_N}\fcont\|^2$ which can be added on both
size of the inequality.
\end{proof}

\section*{Bibliography}

\bibliographystyle{plainnat}

\bibliography{Estim}

\end{document}